\documentclass{amsart}
\RequirePackage[colorlinks,citecolor=blue,urlcolor=blue,linkcolor=blue]{hyperref}
\hypersetup{
colorlinks = true,
citecolor=blue,
urlcolor=blue,
linkcolor=blue,
pdfpagemode = UseNone
}
\usepackage{graphicx,xspace,colortbl,url}
\usepackage{amsmath,amsthm,amsfonts}
\usepackage{color}
\usepackage{epsfig}
\usepackage{subfig}
\usepackage{enumerate}

    \def\qed{\hfill$\sqcap\kern-8.0pt\hbox{$\sqcup$}$\\}
    \def\beq{\begin{eqnarray}}
    \def\eeq{\end{eqnarray}}
    \def\beqq{\begin{eqnarray*}}
    \def\eeqq{\end{eqnarray*}}

    \def\re{\textnormal {Re}}
    
    \def\p{{\mathbb P}}
    \def\e{{\mathbb E}}
    \def\r{{\mathbb R}}
    \def\c{{\mathbb C}}
    \def\pp{{\textnormal p}}	
    \def\d{{\textnormal d}}
    \def\i{{\textnormal i}}

    \def\kk{{\mathcal K}}

    \def\pp{{\mathcal P}} 	
     \def\qq{{\mathcal Q}} 	
    
    \def\vv{{\textnormal v}}
    \def\ee{{\textnormal e}}

\theoremstyle{thmstyleone}
\newtheorem{theorem}{Theorem}[section]
\newtheorem{lemma}[theorem]{Lemma}
\newtheorem{proposition}[theorem]{Proposition}
\newtheorem{corollary}[theorem]{Corollary}
\theoremstyle{thmstylethree}
\newtheorem{definition}{Definition}[section]
\theoremstyle{thmstyletwo}

\newtheorem{remark}[theorem]{Remark}

\usepackage[numbers,sort&compress]{natbib}
\def\topp#1{^{(#1)}}

\newcommand{\uC}{ \mathsf c }

\newcommand{\vA}{  \mathsf a }

\newcommand{\TT}{\mathbb{T}}

\numberwithin{equation}{section}

\newcommand{\la}{\lambda}
\newcommand{\eps}{\varepsilon}

    \newcommand{\comment}[1]{\fbox{\footnotesize WB: \color{magenta}#1}}

\def\vv#1{{\boldsymbol #1}}

\def\topp#1{^{(#1)}}
   
     \newcommand{\Kab}{\mathfrak{K}_{\vA,\uC}\topp \tau}

       \newcommand{\DabLa}{\mathfrak D_{\vA,\uC}\topp \tau}
               \newcommand{\DbaLa}{\mathfrak D_{\uC,\vA}\topp \tau}
   \newcommand{\CabLa}{\mathfrak C_{\vA,\uC}\topp \tau}
\newcommand{\hide}[1]{\comment{\tiny (hidden text)}}

\usepackage{color,graphicx,pgf,amsmath,float,amssymb}
\usepackage{color,amsmath,amssymb}
 \usepackage{framed}
\colorlet{shadecolor}{gray!10}

\newcounter{oldeq}
\newcounter{usesofarxiv}

 \newcommand{\arxiv}[1]{
\setcounter{oldeq}{\value{equation}}
 \addtocounter{usesofarxiv}{1}
 \setcounter{equation}{0}
\def\theoldeq{\theequation}
\def\theequation{x-\arabic{usesofarxiv}.\arabic{equation}}
\def\theequation{\arabic{section}.\arabic{usesofarxiv}.\arabic{equation}}
\def\theequation{\thesection.\arabic{usesofarxiv}.\arabic{equation}}
  \colorlet{shadecolor}{gray!10}
{\footnotesize
\begin{shaded}#1
\end{shaded}
   \setcounter{equation}{\value{oldeq}}
\numberwithin{equation}{section}
}\color{black}}

 \usepackage{tikz}
\usetikzlibrary{patterns}
\usetikzlibrary{positioning}
\usetikzlibrary{shapes.misc}
\usetikzlibrary{shapes.geometric}
\tikzset{cross/.style={cross out, draw=black, fill=none, minimum size=2*(#1-\pgflinewidth), inner sep=0pt, outer sep=0pt}, cross/.default={2pt}}

 \title[Markov processes and stationary measure for the open KPZ]{
Markov processes related to the stationary measure for the open KPZ equation}
\author{W{\l}odek Bryc}
\address
{
W{\l}odzimierz Bryc\\
Department of Mathematical Sciences\\
University of Cincinnati\\
2815 Commons Way\\
Cincinnati, OH, 45221-0025, USA.
}
\email{wlodek.bryc@gmail.com}

\author{Alexey Kuznetsov}
\address
{
Alexey Kuznetsov\\
Department of Mathematics and Statistics\\
York University, 4700 Keele Street
\\ Toronto, Ontario, M3J 1P3, Canada
}
\email{akuznets@yorku.ca}

\author{Yizao Wang}
\address
{
Yizao Wang\\
Department of Mathematical Sciences\\
University of Cincinnati\\
2815 Commons Way\\
Cincinnati, OH, 45221-0025, USA.
}
\email{yizao.wang@uc.edu}
\author{Jacek Weso{\l}owski}
\address{ Jacek Weso{\l}owski\\ Faculty of Mathematics and Information Science\\
Warsaw University of Technology\\
00-662 Warszawa\\
ul. Koszykowa 75 \\
Poland}
\email{wesolo@mini.pw.edu.pl}

\begin{document}\sloppy

\maketitle

\begin{abstract}
We provide a probabilistic
description of the stationary measures
 for the open KPZ on the spatial interval $[0,1]$ in terms of a Markov  process $Y$, which is a Doob's $h$ transform of the Brownian motion killed at an exponential rate. Our work builds on a recent formula of Corwin and Knizel which expresses the multipoint Laplace transform of the stationary solution of the open KPZ in terms of another Markov process $\mathbb T$: the continuous dual Hahn process with Laplace variables taking on the role of time-points in the process.  The core of our approach is to prove that the  Laplace transforms of the  finite dimensional distributions %
of $Y$ and $\mathbb T$ are equal
when the time parameters of one process become the Laplace variables of the other process and vice versa.

\end{abstract}

\arxiv{This is an expanded version of the paper with additional material.}

\section{Introduction and main results}

\subsection{Background}
  The Kardar--Parisi--Zhang (KPZ)
equation is a nonlinear stochastic PDE, proposed in \cite{kardar1986dynamic} as a model   for the evolution of the profile of a growing interface  driven by the space-time white noise $\zeta$. The interface profile  is
described by a height   function $h(t,s)$, where $t\geq 0$ is a time variable and $s$ is a spatial variable, which in one spatial dimension
 formally satisfies a singular equation
\begin{equation}
  \label{KPZ}  \partial_t h(t,s)= \tfrac12 \partial_s^2 h(t,s)+\tfrac12\left(\partial_s h(t,s)\right)^2+\zeta(t,s).
\end{equation}
(For a rigorous interpretation of this equation,  see e.g. \cite{hairer2013solving}.)
In this paper we are interested in the  open KPZ equation on the interval so that \eqref{KPZ} holds
for $0 \leq  s \leq 1$,   subject to    the
inhomogeneous Neumann boundary conditions with two arbitrary parameters.
Since \eqref{KPZ} is not well posed and its solutions are nowhere differentiable, this is only a formal concept.
  Refs. \cite{corwin2018open,gerencser2019singular,parekh2019kpz}, define the solution of this problem in  the Hopf-Cole sense, where the substitution $Z=e^{h}$
 converts \eqref{KPZ} into  the stochastic heat equation
  \begin{equation}
    \label{eq:SHE}
    \partial_t Z(t,s)=\tfrac12 \partial_s^2 Z(t,s)+Z(t,s)\zeta(t,s),\; \; t\geq 0,\; s\in[0,1].
  \end{equation}
  According to \cite[Section 1.1]{CorwinKnizel2021}, the    inhomogeneous   Neumann boundary conditions
  then correspond  to    the inhomogeneous   Robin boundary conditions

  \begin{equation}
    \label{Rbdry}
    \partial_s Z(t,s)\vert_{s=0}=\tfrac{\uC-1}{2}Z(t,0),\quad  \partial_s Z(t,s)\vert_{s=1}=\tfrac{1-\vA}{2}Z(t,1), \; t>0,
  \end{equation}
  with two
    real parameters   $\vA,\uC$.

    The rigorous handling of equation \eqref{eq:SHE} with the boundary conditions \eqref{Rbdry}    is
  technical  and it is not needed in this paper,
  as
we are interested only in the explicit stationary measures
constructed in \cite{CorwinKnizel2021}. %
 A stationary measure for the open KPZ equation
is the law of a %
random function
$( H(s))_{s\in[0,1]}$
with $ H(0)=0$, defined by the property that if %
$h(t,s)$
is a Cole-Hopf solution of equation \eqref{KPZ} with boundary values corresponding to
\eqref{Rbdry} and with the initial condition
$h(0,s) =  H(s)$ for  $s \in  [0,1]$, %
then   for all $t\geq 0$
we have the equality in distribution:
\[
(h(t,s)-h(t,0))_{s\in[0,1]} \stackrel d=(H(s))_{s\in[0,1]}.
\]
As usual, equality in distribution  means that the finite dimensional distributions are the same. (Here,
the related stochastic processes have continuous sample paths, so  equality in  distribution means that they induce the same probability measure on $C([0,1])$.)
Since     we do not need the time variable  for the stationary measure,  we will treat $H$ as a stochastic process on the finite time interval $[0,1]$,
  and
we will  use $t$ instead of $s$  for the spatial variable,   writing  $( H(t))_{t\in[0,1]}$.

 In the seminal study  \cite{CorwinKnizel2021}  Corwin and Knizel
 constructed stationary measures for the open KPZ equation with general parameters $\vA,\uC$,  and determined their Laplace transform when $\vA+\uC> 0$. (More precisely, formula \cite[(1.2)]{CorwinKnizel2021} has two  real parameters $u,v$ that enter the  inhomogeneous Neumann
  boundary conditions; in this paper we  use $\uC=2u$, $\vA=2v$.)
 Transcribed to our notation,   Corwin and Knizel \cite{CorwinKnizel2021} proved that for $0=t_0<t_1<\dots<t_d\leq t_{d+1}=1$ and $\vA+\uC>0$,
    the Laplace transform of the finite-dimensional distributions of $H$
   evaluated at  the increments of a decreasing sequence $S>s_1>s_2>\dots>s_{d}>s_{d+1}=0$, where $S=2+(\uC-2)I_{(0,2)}(\uC)$,   has the form
\begin{multline}
\label{pre-psi}
  \e\left[\exp \left(-\sum_{k=1}^d (s_{k}-s_{k+1})H(t_k)\right)\right]
  \\= \exp\left(\frac14\sum_{k=1}^{d+1}(t_k-t_{k-1})s_k^2\right) \times  \psi\topp {1/4}(\vv s,\vv t),
\end{multline}
where $\vv s=(s_1,\dots,s_d)$ and $\vv t=(t_1,\dots,t_d)$, $d\geq 1$.
To describe  $\psi\topp {1/4}(\vv s,\vv t)$, Corwin and Knizel
 \cite[Section 7.3]{CorwinKnizel2021}
 introduced  the {\em continuous dual Hahn process} %
   $(\mathbb{T}_s)_{s\in[0,S)}$, and   expressed $\psi\topp {1/4}(\vv s,\vv t)$  %
 as %
 \begin{equation}
  \label{psi}
  \psi\topp \tau(\vv s,\vv t)=\frac{1}{\Kab}\int_\r \e\left[\exp\left(-\tau \sum_{k=1}^{d+1}(t_k-t_{k-1})\mathbb{T}_{s_k}\right)\middle\vert\mathbb{T}_{s_{d+1}}=x\right]\mathfrak p_{s_{d+1}}(\d x),
\end{equation}
  where  $\tau=1/4$, constant
  \begin{equation}\label{Kab2p0}
   \Kab =\int_\r e^{-\tau x} \mathfrak p_{s_{d+1}}(\d x)
  \end{equation} is computed from a limiting procedure based on \cite{corwin2018open},   and  $\{\mathfrak p_s(\d x)\}$ is a family of $\sigma$-finite measures \eqref{their-nu} on $\r$, which serve as an entrance law \eqref{p-inv} for the process $\TT$.
  Here we introduced an auxiliary parameter $\tau>0$,
which amounts to replacing   the  interval  $[0,1]$ for the  spatial variable in \eqref{eq:SHE}
by  the interval $[0,4\tau]$.
(We shall  rely on the specific choice $\tau=1/4$ only in
 Section \ref{sec:KPZ}.)

   The family of continuous dual Hahn processes
   is   determined by two parameters $\vA,\uC\in\mathbb R, \vA+\uC>0$. The range of the time parameter $s$ is delicate and we shall discuss it in the paragraph before Theorem \ref{T1.3}.
Transition probabilities
 for $\mathbb T$
 are
 not easy to describe in general, as they may involve both
 an absolutely continuous part and a discrete part, depending on the relation between the location parameter $x$, time parameters $s<t$ and $\vA,\uC$.
 The case $s\in(-\vA,\uC)$ is easier as then,  with state space $(0,\infty)$, the Markov process  $\TT$ has transition   probability density function   \eqref{p2q0}.
For  the more general case we refer to Appendix \ref{Sec:ExtendedCDH}.

  The description    of  the law of $H$ via \eqref{pre-psi} and \eqref{psi} is of an analytical nature and our goal is to provide a probabilistic description. Our starting point is   expression \eqref{psi} for general $\tau>0$. In fact, we shall
show that as a function of $\vv s$, expression \eqref{psi} itself is the Laplace transform in argument $\vv s$,
\begin{equation}
  \label{*post-psi*}
  \psi\topp \tau(\vv s,\vv t) = \e\left[\exp\left(-\sum_{k=1}^d(s_k-s_{k+1})\widetilde Y_{t_k}\right)\right]
\end{equation}
with
\[
(\widetilde Y_t)_{t\in[0,1]} \stackrel d=(Y_0-Y_{\tau t})_{t\in[0,1]}
\]
for a Markov process $(Y_t)_{t\in[0,\tau]}$ from Theorem \ref{T1.1}. (In  general, a product of a Laplace transform with another function $\psi$ can be a Laplace transform without $\psi$ being one.)
As a consequence, taking $\tau=1/4$   we identify   the  stationary measure for the open KPZ equation:
  \begin{equation}
    \label{Y2H}
     (H(t))_{t\in[0,1]} \stackrel d=(B_{t/2}-Y_{t/4}+Y_0)_{t\in[0,1]}.
  \end{equation}
We note that although identity \eqref{Y2H} is rigorously established in Proposition \ref{Prop-KPZ} only for $\tau=1/4$ and $\vA+\uC>0$ with $\vA,\uC>-2$, there are good reasons to expect that it is valid for all $\vA+\uC>0$. %

A pair of formulas \eqref{psi} and \eqref{*post-psi*}
 is a new example of   the {\it dual representation for the Laplace transform}.
 Informally, a dual representation expresses the Laplace transform of the finite-dimensional distribution of the first process as the Laplace transform of the finite dimensional distribution of the second process  when the time parameters of one process become the Laplace variables of the other process and vice versa.
Ref. \cite{Bryc-Wang-2017} established such   dual  representations for the Laplace transforms of Markov processes which are squares of  the  radial  part  of  a  3-dimensional  Cauchy  process \cite[Corollary 1]{KyprianouOConnell2021},
and  the Laplace transforms of the Brownian excursion and the generalized meanders. Such identities were needed to identify limiting fluctuations for the open ASEP in \cite{Bryc-Wang-2017ASEP}. A different (univariate) example of such a dual representation  is %
  \cite[formula (2)]{Bertoin-Yor2001subordinators}.  The specifics differ slightly from case to case, compare   \eqref{Dual-1}, \eqref{For-proof}, and \cite[(1.2)]{Bryc-Wang-2017}.
 At this time there is no complete understanding
how such dual representations may arise in general.

\subsection{Main results}
We first introduce the Markov process $(Y_t)_{t\in[0,\tau]}$ (with parameters $\vA,\uC,\tau$)
based on the Yakubovich heat kernel
  \begin{equation}\label{p_t}
p_t(x,y)=  \int_0^{\infty} e^{-t u^2} K_{\i u}(e^x) K_{\i u}(e^y)\, \mu(\d u),\quad x,y\in\r,\; t>0,
\end{equation}
where
\begin{equation}\label{mu}
\mu(\d u):= \frac{2}{\pi}\frac{\d u}{\vert \Gamma(\i u)\vert ^2}=\frac{2}{\pi^2} u\sinh(\pi u)\, \d u,
\end{equation}
and
 \begin{equation}
  \label{K-def}
  K_{\i u}(x)=\int_0^\infty e^{-x\cosh w}\cos(u w)\, \d w
\end{equation}
is the modified Bessel function defined for $u\in\r$ and  $x>0$.

 The following result specifies the finite dimensional distributions of the process, {which is all that is needed for the dual representation of the Laplace transforms.}

\begin{theorem}\label{T1.1}
  Fix $\tau>0$ and $\vA,\uC\in\r$ such that $\vA+\uC>0$. Then there exists a Markov process $(Y_t)_{t\in [0,\tau]}$ such that
for $0=t_0<t_1<\dots<t_{d}<t_{d+1}=\tau$ and $d\geq 0$   the  finite dimensional distribution of
$(Y_{t_0},\ldots,Y_{t_{d+1}})$
has the %
 density
     \begin{equation}\label{joint-density}
   f_{\vv t}(\vv x)=\frac{1}{ C\topp \tau_{\vA,\uC}} e^{\uC x_0 + \vA x_{d+1}}\prod_{k=1}^{d+1}p_{t_k-t_{k-1}}(x_{k-1},x_{k}),
 \end{equation}
 where $\vv t=(t_0,\dots,t_{d+1})$, $\vv x=(x_0,\dots,x_{d+1})\in\r^{d+2}$ and %
  \begin{equation}
   \label{C-formula}
    C\topp \tau_{\vA,\uC}:=\int_{\r} \int_{\r} e^{\vA x+\uC y} p_\tau(x,y)\d x \d y.
 \end{equation}

\end{theorem}
The proof of  Theorem \ref{T1.1} consists of the construction
of the Markov process, and includes verification that  $C\topp \tau_{\vA,\uC}<\infty$.

Now, we present the dual representation for the Laplace transform for the processes $\TT$ and $Y$ when the process  $\TT$ is restricted to $s\in(-\vA,\uC)$. This simplifies calculations as  then transition probabilities for $\TT$ are absolutely continuous with probability density function
\begin{equation}
  \label{p2q0}
  \mathfrak p_{s,t}(x,\d y) %
 = \frac{\vert \Gamma(\frac{\uC-t+\i \sqrt{y}}{2},\frac{t-s+\i (\sqrt{x}+\sqrt{y})}{2},\frac{t-s+\i (\sqrt{x}-\sqrt{y})}{2})\vert ^2}{8\pi\sqrt{y}\Gamma(t-s)\vert \Gamma(\frac{\uC-s+\i \sqrt{x}}{2},\i \sqrt{y})\vert ^2}1_{\{y>0\}}\d y,\;\;
 s<t,\; x>0.
\end{equation}
Here and throughout  the paper we use  standard notation for the products of Gamma functions:
$$\Gamma(a_1,a_2,\dots,a_k)= \Gamma(a_1)\Gamma(a_2)\cdots\Gamma(a_k).$$
\begin{theorem}\label{T1.2} For $\vA+\uC>0$,  $d= 0,1, \dots$,
$\uC>s_1>\dots>s_d>s_{d+1}>-\vA$,
and $t_0=0<t_1<\dots<t_d<t_{d+1}=\tau$, %
\begin{multline}\label{Dual-1}
   \int_0^\infty \e\left[\exp\left(-\sum_{k=1}^{d+1}(t_k-t_{k-1})\TT_{s_k}\right)\middle\vert \TT_{s_{d+1}}=x\right]
  \widetilde {\mathfrak p}_{s_{d+1}}(x)\d x
   \\ =
   \e\left[ \exp\left(\sum_{k=1}^{d+1}  s_{k} (Y_{ t_{k}}-Y_{t_{k-1}}) \right)\right],
\end{multline}
where  $(Y_t)_{t\in[0,\tau]}$ is the Markov process from Theorem \ref{T1.1} and for $-\vA<s<\uC$,
\begin{equation}
  \label{tilde-mathfrak-p}\widetilde {\mathfrak p}_s(x)=\frac{2^{\vA+\uC}}{ 16\pi C\topp \tau_{\vA,\uC}}\frac{\vert \Gamma(\tfrac{\vA+s+\i \sqrt{x}}{2},\tfrac{\uC-s+\i \sqrt{x}}{2})\vert ^2}{\sqrt{x}\vert \Gamma(\i \sqrt{x})\vert ^2},\; x>0.
\end{equation}
\end{theorem}

The left hand side of \eqref{Dual-1} is similar to the expression on the right hand side of \eqref{psi}, but \eqref{Dual-1} does not lead to    \eqref{*post-psi*} immediately. Firstly,
we have $s_{d+1}=0$ in \eqref{psi}, while $s_{d+1}$ is not necessarily zero in \eqref{Dual-1}: in fact, $s_{d+1}=0$ is not allowed in \eqref{Dual-1} if $\vA\le 0$.
 Secondly,  for technical reasons, the range of the time parameter in \eqref{psi}  is restricted in \cite{CorwinKnizel2021} to $s\in[0,S)$ only.
 Notice that  the process $\mathbb T$ over $[0,S)$ may have atoms in its law (see \citep{CorwinKnizel2021} for a detailed presentation).
Moreover, the intervals $[0,S)$ and $(-\vA,\uC)$ may be disjoint. Nevertheless, when $\uC>0$ (so $S = \min\{\uC,2\}$),  the continuous dual Hahn process $\TT$ can be actually defined over a larger domain $(-\infty,\uC)\supset[0,S)$, as shown recently in \cite{Bryc-2021} and summarized in Theorem \ref{TH:CDH}.
 Thirdly, we need to
relate the normalizing constants \eqref{Kab2p0}  and \eqref{C-formula}.

These issues are addressed in our next result.

\begin{theorem}\label{T1.3}
Fix $\tau>0$,  $d\geq 1$,  and real $\vA,\uC$ such that  $\uC>\max\{-\vA,0\}$. Let $t_0=0<t_1<\dots<t_d=1$
and $\uC>s_1>\dots>s_{d}>\max\{-\vA,0\}$, $s_{d+1} =0$. Then
\begin{multline}\label{psi-dual}
   \frac{1}{\Kab}\int_\r \e\left[\exp\left(-\tau \sum_{k=1}^{d}(t_k-t_{k-1})\mathbb{T}_{s_k}\right)\middle\vert \mathbb{T}_{0}=x\right]\mathfrak p_{0}(\d x)
   \\=
  \e\left[ \exp\left(\sum_{k=1}^{d}  -(s_{k}-s_{k+1}) (Y_0- Y_{ \tau t_{k}})  \right)\right],
\end{multline}
where $\Kab$ is given by \eqref{Kab2p0}, measure $\mathfrak p_{0}(\d x)$ is defined in
\eqref{their-nu},  process
$(\TT_s)_{s\in(-\infty,\uC)}$ is the continuous dual Hahn process, see Definition \ref{Def.C1}, and $(Y_t)_{t\in[0,\tau]}$ is the Markov process introduced in  Theorem \ref{T1.1}.  \end{theorem}
We emphasize again that, although the left-hand side of \eqref{psi-dual} has the identical expression as $\psi\topp\tau(\vv s,\vv t)$ in \eqref{psi}, the admissible range of time parameters $\vv s$ in both formulas is not the same.
 So when we apply Theorem \ref{T1.3}   in the proof of Proposition \ref{Prop-KPZ}, we need to add an assumption that $\vA>-2$  to ensure that $[0,S)\cap ( -\vA,\uC)\ne\emptyset$.

\subsection{Application to  stationary measure for the open KPZ equation} \label{sec:KPZ}
Recall that  we denoted by $(H(t))_{t\in[0,1]}$  a process with the
stationary measure
for the open KPZ equation as its law.  In \cite[(1.12)]{CorwinKnizel2021} this process was denoted by $(H_{u,v}(x))_{x\in[0,1]}$  with $u=\uC/2$, $v=\vA/2$.
We have the following description of the law of $H$.

\begin{proposition}\label{Prop-KPZ}
Let $(Y_t)_{t\in[0,1/4]}$ be the Markov process  from Theorem \ref{T1.1} with parameter
 $\tau=1/4$, and let $(B_t)_{t\geq 0}$ be  an independent
Brownian motion. If $\vA+\uC>0$ with $\min\{\vA,\uC\}>-2$, then
 \eqref{Y2H} holds.

\end{proposition}

\begin{proof}
   We first consider  the case $\uC>0$.  Then $S = \min\{\uC,2\}$.
   We  apply Theorem \ref{T1.3} with the sequence $s_1>\dots>s_d$ from the  interval $(\max\{-\vA,0\},S)$, which by assumption  is non-empty,
and with $\tau=1/4$.
On this interval, \eqref{pre-psi},  \eqref{psi} and \eqref{psi-dual} hold.
It is readily checked that the first factor on the right-hand side of \eqref{pre-psi} corresponds to the Laplace transform of the increments of the scaled Brownian motion as desired.
Since  %
 $s_1>\dots>s_d>s_{d+1}:=0$,
 taking $t_d=t_{d+1}=1$ in \eqref{pre-psi} and  \eqref{psi},
  we get
\begin{multline}\label{H2Y}
    \e\left[e^{ -\sum_{k=1}^d (s_{k}-s_{k+1})H(t_k)}\right]  \stackrel{\eqref{pre-psi}}{=}
       \e\Big[e^{-\sum_{k=1}^{d} s_k(B_{t_k/2}-B_{t_{k-1}/2})}\Big]\psi\topp{1/4}(\vv s,  \vv t)
 \\\stackrel{\eqref{psi}}{=} \e\Big[e^{-\sum_{k=1}^{d} (s_k-s_{k+1})B_{t_k/2}}\Big]\\ \times\frac{1}{\mathfrak{K}_{\vA,\uC}\topp {1/4}}\int_\r \e\left[\exp\left(-\tfrac14 \sum_{k=1}^{d}(t_k-t_{k-1})\mathbb{T}_{s_k}\right)\middle\vert \mathbb{T}_{0}=x\right]\mathfrak p_{0}(\d x)
 \\
  \stackrel{\eqref{psi-dual}}{=}
   \e\Big[ e^{  \sum_{k=1}^{d}  -(s_{k}-s_{k+1}) (B_{t_k/2} -Y_{t_{k}/4}+Y_{0})} \Big].
\end{multline}
\arxiv{Here is the identity \eqref{H2Y}  with the intermediate steps:
\begin{multline}\label{H2Y*}
    \e\left[e^{ -\sum_{k=1}^d (s_{k}-s_{k+1})H(t_k)}\right]  \stackrel{\eqref{pre-psi}}{=}
       \e\Big[e^{-\sum_{k=1}^{d} s_k(B_{t_k/2}-B_{t_{k-1}/2})}\Big]\psi\topp{1/4}(\vv s,  \vv t)
 \\=\e\Big[e^{-\sum_{k=1}^{d} (s_k-s_{k+1})B_{t_k/2}}\Big] \psi\topp{1/4}(\vv s,  \vv t)
 \\\stackrel{\eqref{psi}}{=} \e\Big[e^{-\sum_{k=1}^{d} (s_k-s_{k+1})B_{t_k/2}}\Big]\frac{1}{\mathfrak{K}_{\vA,\uC}\topp {1/4}}\int_\r \e\left[\exp\left(-\tau \sum_{k=1}^{d+1}(t_k-t_{k-1})\mathbb{T}_{s_k}\right)\middle\vert \mathbb{T}_{s_{d+1}}=x\right]\mathfrak p_{s_{d+1}}(\d x)\\
 \\= \e\Big[e^{-\sum_{k=1}^{d} (s_k-s_{k+1})B_{t_k/2}}\Big]\frac{1}{\mathfrak{K}_{\vA,\uC}\topp {1/4}}\int_\r \e\left[\exp\left(-\tau \sum_{k=1}^{d}(t_k-t_{k-1})\mathbb{T}_{s_k}\right)\middle\vert \mathbb{T}_{0}=x\right]\mathfrak p_{0}(\d x)
 \\
  \stackrel{\eqref{psi-dual}}{=}
   \e\Big[ e^{  \sum_{k=1}^{d}  -(s_{k}-s_{k+1}) (B_{t_k/2} -Y_{t_{k}/4}+Y_{0})} \Big].
\end{multline}
}
Using the fact that the $d$-tuples $(s_1-s_2,s_2-s_3,\dots,s_{d-1}-s_{d},s_d)$  can be selected arbitrarily from an open subset of $\r^d$,  the first and last expressions in \eqref{H2Y}   identify  the laws  by  uniqueness of  the Laplace transform \cite[Theorem 2.1]{farrell2006techniques}. This proves  \eqref{Y2H} %
when $\uC>0$.

By symmetry, this identification  extends to the case when $\uC$ is negative with $\uC>-2$. For this argument we need to indicate
explicitly how the processes depend on parameters  $\vA,\uC$,
so
by $Y\topp{{\vA,\uC}}$ we denote the process with joint density \eqref{joint-density}
 with $\tau=1/4$,
and by $H\topp{{\vA,\uC}}$ we denote the process $(H_{\uC/2,\vA/2}(x))_{x\in[0,1]}$ in \cite[(1.12)]{CorwinKnizel2021}.

The argument is based on the fact that time reversal exchanges the roles of parameters $\vA,\uC$.  From \eqref{joint-density}  and \eqref{p_t} (recall that $\tau=1/4$) it follows that  process $(Y\topp{{\vA,\uC}}_{(1-t)/4})$ has the same law as the process $(Y\topp{\uC,\vA}_{t/4})$,
and
   by \cite[Theorem 1.4 (4)]{CorwinKnizel2021},  process $(H\topp{\vA,\uC}(1-t)-H\topp{\vA,\uC}(1))_{t\in[0,1]}$ has the same law as
   $(H\topp{\uC,\vA}(t))_{t\in[0,1]}$.
  So if $\uC$ is negative, i.e., $\vA>0$,   applying the already established result to the time reversals, we see that  process
  $(H\topp{{\uC,\vA}}(t))_{t\in[0,1]}$ has the same law as the process $(\widetilde H(t))_{t\in[0,1]}$ with $\widetilde H(t)=B_{t/2}-Y_{t/4}\topp{{\uC,\vA}}+Y_0\topp{{\uC,\vA}}$. Applying the time reversals again, we see that
  process
  $(H\topp{{\vA,\uC}}(t))_{t\in[0,1]}$ has the same law as the process  defined by the reverse transformation of $\widetilde H$, i.e.,
  \begin{align*}
     \widetilde H(1-t)-\widetilde H(1) & =
   (B_{(1-t)/2}-Y\topp{\uC,\vA}_{1/4-t/4}+Y\topp{\uC,\vA}_{0})-(B_{1/2}-Y\topp{\uC,\vA}_{1/4}+Y\topp{\uC,\vA}_{0}) \\& =
    (B_{(1-t)/2}-B_{1/2})+(Y\topp{\uC,\vA}_{1/4}-Y\topp{\uC,\vA}_{1/4-t/4}).
  \end{align*} This ends the proof, as
    $(B_{(1-t)/2}-B_{1/2})_{t\in[0,1]}\stackrel{d}= (B_{t/2})_{t\in[0,1]}$ and $(Y\topp{\uC,\vA}_{1/4}-Y\topp{\uC,\vA}_{1/4-t/4})_{t\in[0,1]}\stackrel{d}=
    (Y\topp{\vA,\uC}_{0}-Y\topp{\vA,\uC}_{t/4})_{t\in[0,1]}$.
\end{proof}

 \subsection*{Note} %
 After the first version of this paper was posted,     Barraquand and Le Doussal posted preprint of Ref. \cite{barraquand2021steady} in which they obtained a different representation of  the stationary measure of the KPZ equation on an interval as
  \begin{equation}
    \label{X2H} \left( H(t)\right)_{t\in[0,1]}\stackrel{d}{=}\left(B_{t/2}+X_t\right)_{t\in[0,1]},
  \end{equation}
     where $(B_t)_{t\in[0,1]}$ is the Brownian motion as before,
  $(X_t)_{t\in[0,1]}$   is an independent  stochastic process  with continuous trajectories such that the Radon-Nikodym derivative of its law $ \p_X$ on $C[0,1]$ with respect to the law $\p_{\widetilde B}$ of the scaled Brownian motion $({B}_{t/2})_{t\in[0,1]}$ is
  \begin{equation}\label{eq:Bar-LeD}
   \frac {\d \p_X}{\d \p_{\widetilde B}}
 =\frac{1}{ \mathsf Z} e^{-\vA \beta_1} \left(\int_0^1 e^{-2 \beta_t}dt\right)^{-\vA/2-\uC/2}
 \end{equation}
with the argument of %
the
Radon-Nikodym derivative
 denoted by $\beta=(\beta_t)_{t\in[0,1]}\in C[0,1]$. (Here $ \mathsf Z$ is a normalization constant.)
The proof in \cite{barraquand2021steady}  is given  for $\vA,\uC>0$, but it is conjectured that \eqref{X2H} represents the    stationary measure for the KPZ equation on any finite interval  with the Neumann boundary  conditions for any real $\vA,\uC$.

It was noted in   \cite{barraquand2021steady}  that $(X_t)_{t\in[0,1]}\stackrel d=(Y_0-Y_{t/4})_{t\in[0,1]}$  when $\vA,\uC>0$, and this result was later extended in \cite{Bryc-Kuznetsov-2021}  to all $\vA+\uC>0$.
To point out the difference between \eqref{X2H} and \eqref{Y2H},  we note that
our result does not extend to the case $\vA+\uC=0$, as the
 Markov process $(Y_t)_{t\in[0,\tau]}$ from Theorem \ref{T1.1}  is then not well defined.
On the other hand,  formula \eqref{eq:Bar-LeD} in this case gives the Brownian motion with drift, so \eqref{X2H}  gives the same law as  in \cite{CorwinKnizel2021}.

\begin{remark}
  \label{Rem-extend}
In view of  \cite[Remark 8.2]{CorwinKnizel2021}, it is plausible that representation \eqref{Y2H} holds %
 in the entire range of parameters $\vA+\uC>0$.
A closely related remark is that the restriction $s_d>\max\{-\vA,0\}$ in Theorem \ref{T1.3} might be relaxed to $s_d>0$ by  an analytic continuation argument as in the proof of Theorem \ref{Thm-L-Const} below.
 We do not pursue this here, as this would require consideration of atoms in the transition probabilities for the dual continuous Hahn process $(\mathbb{T}_s)$.
\end{remark}

\begin{remark}
Expression \eqref{psi} can also be written in the following %
form:
\begin{equation*}\label{Jacek-psi}
\psi\topp \tau(\vv s,\vv t) =\e\left[\exp\left( -\tau \sum_{k=1}^{d} t_k (\mathbb{T}_{s_k}-\mathbb{T}_{s_{k+1}})\right) \right],
\end{equation*}
where $(\mathbb T_s)_{s\in[0,S)}$ is the continuous dual Hahn process  with transition probabilities
\eqref{T-trans} and with the initial law given by the probability measure
\begin{equation*}\label{NEF}
    \p(\mathbb T_0=\d x)= \frac{1}{\Kab} e^{-\tau x} \mathfrak p_{0}(\d x).
\end{equation*}
\end{remark}

 \begin{remark}
 It might be interesting to point out  that identity \eqref{Dual-1} can  formally be written in the integral form.
 For %
  $\uC>s_1>\dots>s_d>s_{d+1}>-\vA$,
 and $t_0=0<t_1<\dots<t_d<t_{d+1}=\tau$
 define the
 step-wise function %
 \begin{equation*}
   \label{step-s}
   \sigma (t)=\sum_{j=1}^{d+1} s_j \mathbf{1}_{(t_{j-1},t_{j}]}(t).
 \end{equation*}
 By linearity,  the integral of   $\sigma$ is well defined for any random signed measure on a field generated  by the left-open right-closed intervals. An example of such a measure is  $(\alpha,\beta]\mapsto Y_\beta-Y_\alpha$ induced by process $(Y_t)_{t\in[0,\tau]}$.
  Then   %
  \eqref{Dual-1} can be written as
 \begin{multline*}\label{IbyP}
\int_0^\infty \e\left[ \exp\left(- \int_0^\tau \TT_{\sigma(t)} \d t \right)\middle \vert  \TT_{s_{d+1}}=x\right] \widetilde{\mathfrak p}_{s_{d+1}}(x)\d x\\= \e \left[\exp\left(\int_0^\tau  \sigma(t)\d Y_t\right)\right].
 \end{multline*}
 A related integral expression is discussed in \cite[Section 2.1.4]{CorwinKnizel2021}.
 \end{remark}

\subsection{Overview of the paper}  %
We begin with  a short self-contained proof of Theorem \ref{T1.2}  in Section~\ref{Sect:SecondProof}.
In Section \ref{Sect:Yakubovich} we introduce a non-standard version of the Kontorovich-Lebedev transform,   discuss   the Yakubovich heat kernel \eqref{p_t} and its semigroup,
and give an  analytic continuation argument that we use to relate the normalizing constants
\eqref{Kab2p0} and \eqref{C-formula}.
In Section \ref{Sect:Recipr}  we  prove %
Theorem \ref{T1.1}.
In Section \ref{Sec:ProofT1.3}  we prove Theorem \ref{T1.3}.

In Section \ref{Sect: DMS} we
 use  the  Kontorovich-Lebedev transform  to prove   two dual representations, one conditional and one unconditional.
 The unconditional dual representation in Theorem \ref{T2}  yields Theorem \ref{T1.2}, but we suppressed that argument in lieu of the more  direct ad-hoc argument  in Section \ref{Sect:SecondProof}.

 In Section \ref{Sect: Rel-HW} we discuss the Hartman-Watson density  and  its relation to the  Yakubovich heat kernel. Using this relation, we derive a closed form expression  for the Laplace transform,  with respect to argument $\tau$, for the normalizing constant \eqref{C-formula} in the strip $0<\vA+\uC<2$.

 In the Appendix we   collect some  known results  in the form that we need, sometimes with sketches of proofs for completeness.

\section{Proof of Theorem \ref{T1.2}} \label{Sect:SecondProof}
Let  $(Z_s)_{s\in(-\infty, \uC)}$ be a Markov
process
with state space  $(0,\infty)$ and transition probabilities %
$\mathbb P(Z_t = \d v\mid Z_s = u)$
  given by the densities
\begin{equation}\label{ICAK-q}
   q_{s,t}(u,v)= \frac{\vert \Gamma(\frac{\uC-t+\i v}{2},\frac{t-s+\i (u+v)}{2},\frac{t-s+\i (u-v)}{2})\vert ^2}{4\pi\Gamma(t-s)\vert \Gamma(\frac{\uC-s+\i u}{2},\i v)\vert ^2},\;\; -\infty< s<t<\uC, \; \; u,v>0 .
\end{equation}
We are going to prove that
\begin{multline}\label{ICAK-A1}
   \int_0^\infty \e\left[\exp\left(-\sum_{k=1}^{d+1}(t_k-t_{k-1})Z_{s_k}^2\right)\middle\vert Z_{s_{d+1}}=u\right]\varphi_{s_{d+1}}(u)\d u
  \\ =
   \e\left[ \exp\left(\sum\limits_{k=1}^{d+1}  s_{k} (Y_{ t_{k}}-Y_{t_{k-1}}) \right) \right],
\end{multline}
where for $-\vA<s<\uC$,
\begin{equation}
  \label{nu_s}
  \varphi_s(u)=\frac{2^{\vA+\uC}}{ 8\pi C\topp \tau_{\vA,\uC}}\frac{\vert \Gamma(\tfrac{\vA+s+\i u}{2},\tfrac{\uC-s+\i u}{2})\vert ^2}{\vert \Gamma(\i u)\vert ^2}.
\end{equation}
Comparing \eqref{ICAK-A1} and \eqref{p2q0}, we see that  $Z_s^2=\TT_s$.
Since
$2\sqrt{x}\,\widetilde {\mathfrak p}_s(x)= \varphi_s(\sqrt{x}) $,
after a change of variable $u^2=x$  the left hand sides of \eqref{ICAK-A1} and \eqref{Dual-1} are the same. So the proof of Theorem \ref{T1.2}
 reduces to the proof of
 \eqref{ICAK-A1}.
\begin{proof}[Proof of Theorem \ref{T1.2}]
This proof consists of verification of the integral identity \eqref{ICAK-A1}, assuming Theorem \ref{T1.1}.  The two key identities are
\begin{equation}  \label{K-Melin}
\int_{\r } e^{t x} K_{\i u}(e^x) K_{\i v} (e^x) \d x=\frac{2^{t-3}}{\Gamma(t)}
\vert \Gamma((t+\i (u+v))/2,(t+\i (u-v))/2)\vert ^2,
\end{equation}
which holds for $u,v\in\r$ and $\re(t)>0$ (see {\cite[6.8 (48)]{erdelyi1954fg}}), and
\begin{equation}\label{K-Mellin2}
2^{s-2} \vert \Gamma(\tfrac{s+\i u}{2})\vert ^2=
\int_{\r} e^{sx} K_{\i u}(e^x)\d x,
\end{equation}
which holds for all $u\in\r$ and $\re(s)>0$ (see %
 \cite[6.8 (26)]{erdelyi1954fg}).

For ease of reference, we denote the left-hand side of \eqref{ICAK-A1} by
$\Psi(\vv s,\vv t)$
  and  expand it  as an  explicit multivariate integral %
    \begin{multline*}
 \Psi(\vv s,\vv t)= \\
   \int_0^\infty \varphi_{s_{d+1}}(u_{d+1}) \d u_{d+1} \int_{\r_+^d} e^{-  \sum_{k=1}^{d+1}(t_k-t_{k-1})u_k^2} \prod_{k=1}^{d} q_{s_{k+1},s_{k}}(u_{k+1},u_k) d u_1\dots\d u_d,
\end{multline*}
where  $\uC>s_1>\dots>s_d>s_{d+1}>-\vA$ with  $\vA+\uC>0$, and $q_{s,t}(u,v)$ is given by \eqref{ICAK-q}.
Denote $\Delta_k=t_{k}-t_{k-1}$. Then

\begin{align}\label{Yizao} C\topp \tau_{\vA,\uC}\,\Psi(\vv s,\vv t) & = \tfrac{2^{a+c}}{16}\,\int_{\r_+^{d+1}}\exp\left(- \sum_{k=1}^{d+1}u_k^2\Delta_k\right)
\left\vert \Gamma\left(\tfrac{\vA+s_{d+1}+\i u_{d+1}}{2},\tfrac{\uC-s_{1}+\i{u_{1}}}{2}\right)\right\vert ^2\\
& \quad \times\prod_{k=1}^{d}\tfrac{\left\vert \Gamma\left(\tfrac{s_k-s_{k+1}+\i( u_k-u_{k+1})}{2}, \tfrac{s_k-s_{k+1}+\i(u_k+u_{k+1})}{2}\right)\right\vert ^2}{8\Gamma(s_k-s_{k+1})}\,\mu(\d\vv u),\nonumber
\end{align}
where $\mu(\d\vv u)=\prod_{k=1}^{d+1}\,\mu(\d u_k)$. Inserting \eqref{K-Mellin2} twice and \eqref{K-Melin} $d$ times into \eqref{Yizao}, we get
\footnotesize
\begin{multline}  C\topp \tau_{\vA,\uC}\, \Psi(\vv s,\vv t)=
   \\  \int_{\r_+^{d+1}}\,e^{- \sum_{k=1}^{d+1}u_k^2\Delta_k}\;{\int_\r e^{(\vA+s_{d+1}) x_{d+1}}K_{\i u_{d+1}}(e^{x_{d+1}})\tfrac{\d x_{d+1}}{2^{s_{d+1}-s_1}} \int_\r  e^{(\uC-s_1)x_{0}} K_{\i u_1}(e^{x_{0}})\d x_{0}} \label{MultiIntegrals}\\
  \quad \times\left(\prod_{k=1}^{d}\tfrac{1}{ 2^{s_k-s_{k+1}}} \int_\r e^{(s_k-s_{k+1})x_k}
K_{\i u_k}\left(e^{x_k}\right)
K_{\i u_{k+1}}\left(e^{x_k}\right)\d x_k\right)\, \mu(\d\vv u)  =
\\
  \int_{\r_+^{d+1}}\,e^{- \sum_{ k=1}^{d+1}u_k^2\Delta_k}\;\int_\r e^{(\vA +s_{d+1})x_{d+1}} K_{\i u_{d+1}}\left(e^{x_{d+1}}\right)\d x_{d+1} \int_\r  e^{(\uC-s_1)x_{0}} K_{\i u_1}\left(e^{x_{0}}\right)\d x_{0} \\
  \quad \times\left(\prod_{k=1}^{d} \int_\r e^{(s_k-s_{k+1})x_k} K_{\i u_k}\left(e^{x_k}\right)
K_{\i  u_{k+1}}\left(e^{x_k}\right) \d x_k\right)\,\mu(\d \vv u). \end{multline}
\normalsize
Recalling \eqref{mu} and \eqref{K-def},  the expression under the multiple integrals in \eqref{MultiIntegrals}  is bounded, up to a multiplicative constant,  by the product of two integrable expressions:
$$ e^{- \sum_{ k=1}^{d+1}u_k^2\Delta_k},
$$
which is integrable with respect to $\mu(\d\vv u)$ as $\Delta_k=t_k-t_{k-1}>0$, and
$$
e^{(\vA+s_{d+1}) x_{d+1}} K_0\left(e^{x_{d+1}}\right)e^{(\uC-s_1)x_0} K_0\left(e^{x_{0}}\right)\prod_{k=1}^d e^{(s_k-s_{k+1})x_k} K_0^2\left(e^{x_k}\right),
$$
which is integrable over $\r^{d+2}$ with respect to $\d\vv x$, since functions $e^{\eps x}K_0(e^{x})$ and  $e^{\eps x}K_0^2(e^{x})$ are integrable for any $\eps>0$ by the well known bounds \eqref{K0-bd1} and \eqref{K0-bd2}.
Using Fubini’s theorem, we rearrange the order of integrals in \eqref{MultiIntegrals}. We get
\begin{multline*} %
 C\topp \tau_{\vA,\uC} \Psi(\vv s,\vv t) =
   \int_{\r^{d+2}}e^{(\uC-s_1)x_{0} +\sum_{k=1}^d(s_k-s_{k+1})x_k +(\vA +s_{d+1})x_{d+1}}
  \\
\times \int_{\r_+^{d+1}}e^{- \sum_{ k=1}^{d+1}u_k^2\Delta_k}  K_{\i u_{d+1}}\left(e^{x_{d+1}}\right)    K_{\i u_1}\left(e^{x_{0}}\right)
\\\times  \left(\prod_{k=1}^{d}   K_{\i u_k}\left(e^{x_k} \right)
K_{\i  u_{k+1}}\left(e^{x_k}\right)\right)\,\mu(\d \vv u)\,
\d \vv x.
\end{multline*}
Noting that
$$
(\uC-s_1)x_{0} +\sum_{k=1}^d(s_k-s_{k+1})x_k +(\vA +s_{d+1})x_{d+1}=
\uC x_0+\vA x_{d+1}+\sum_{k=1}^{d+1} s_k(x_k-x_{k-1}),$$
we have
\begin{align*} %
 C\topp \tau_{\vA,\uC}
 \Psi(\vv s,\vv t)& =
  \int_{\r^{d+2}}e^{\sum_{k=1}^{d+1} s_k(x_k-x_{k-1})}e^{\uC x_{0} + \vA  x_{d+1}}
  \\
& \quad  \times
 \left(\prod_{k=1}^{d+1}\int_{\r_+}e^{-  u_k^2(t_k-t_{k-1})}    K_{\i u_k}\left(e^{x_k} \right)  K_{\i u_k}\left(e^{x_{k-1}} \right)   \mu(\d u_k)\right)
  \d \vv x
\\ & =  C\topp \tau_{\vA,\uC} \int_{\r^{d+2}} e^{\sum_{k=1}^{d+1} s_k(x_k-x_{k-1})}\,
f_{\vv t}(\vv x) \d \vv x,
\end{align*}
where the last formula follows from \eqref{p_t} and \eqref{joint-density}. Thus \eqref{ICAK-A1} holds.
\end{proof}
\section{Killed Brownian motion and the Yakubovich heat kernel}
 \label{Sect:Yakubovich}
 In this section we introduce   notation, collect results from the literature, and prove auxiliary facts that we need for  Section \ref{Sect:Recipr}, where we prove Theorem \ref{T1.1} and for Section \ref{Sec:ProofT1.3}, where we prove Theorem \ref{T1.3}.
\subsection{Kontorovich--Lebedev transform} \label{Sec:Bkg}

According to \cite[10.45(v)]{NIST2010}  (see also the papers \cite{Yakubovich2004,Yakubovich2011}, or the monograph \cite{yakubovich1996index}),
the Kontorovich--Lebedev integral transform is defined as %
\begin{equation*}
\mathbb K_{\i u} [f]:=\int_0^{\infty} K_{\i u} (x) f(x) \d x.
\end{equation*}
It is an isometry between $L_2((0,\infty), x\d x)$ and $L_2((0,\infty), \mu(\d u))$, where $\mu$ is given by \eqref{mu}.

If $g(u)=\mathbb K_{\i u} [f]$, then the inverse transform is
\begin{equation*}
f(x)=\frac{1}{x} \int_0^{\infty} K_{\i u} (x) g(u) \mu(\d u),
\end{equation*}
see \cite[Section 4.15]{titchmarsh2011elgenfunction} and \cite{Yakubovich2004,Yakubovich2011}.

We will use a slightly different version of the Kontorovich--Lebedev transform, defined as
\begin{equation}\label{K-L+}
{\mathcal K}f (u)=\int_{\r} f(x) K_{\i u}(e^x) \d x.
\end{equation}
Note the difference between  $K$,  $\mathbb K$,  and curly ${\mathcal K}$ in the notation.
  From the above facts we conclude that ${\mathcal K}$ is an isometry between $L_2(\r, \d x)$ and $L_2((0,\infty), \mu(\d u))$ and the inverse transform is given by
\begin{equation}\label{K-inv}
{\mathcal K}^{-1} g(x)= \int_0^{\infty} K_{\i u}(e^{x}) g(u) \mu(\d u).
\end{equation}
Note that if $f$ is a function of $x\in\r$ and $g$ is a function of $u>0$, then  $\kk f$ is a function of $u>0$ and $\kk^{-1}g$ is a function of $x\in\r$.
When we need to apply $\kk$ to an  explicit expression in variable $x$, we will write $\kk[f(x)]$ or $\kk[f(x)](u)$. In this notations, the left hand side of \eqref{K-L+} is an abbreviated form of the more precise expression $\kk[f(x)](u)$. We will use similar  conventions in our notation for the operators. %

  As integral operators, both  $\mathcal{K}$ and $\kk^{-1}$   extend to functions which are not necessarily square integrable, and we will need to apply them to such functions.
  (Even though they may cease to be inverses of each other.)
  It is clear that $\mathcal{K}$  is well defined on
$L_1(\r, K_0(e^x)\d x)$ which we will abbreviate as $L_1(K_0)$,
compare
 \cite[(1.2)]{Yakubovich2003kontorovich}. It is clear that $\kk^{-1}$ is well defined on  $L_1((0,\infty),\d \mu)$, which we will abbreviate as $L_1(\d \mu)$.

 \subsection{The   semigroup with kernel $p_t(x,y)$}
 Recall definition \eqref{p_t}. We can now explain how Yakubovich heat kernel  $p_t(x,y)$ defines  a sub-Markov process. %
 (Note that process $(Y_t)$ in Theorem \ref{T1.1} is Markov, not sub-Markov.)
 The Kontorovich--Lebedev transform defines a semigroup of contractions %
 \begin{equation}\label{PP}
\widetilde {\pp}_t ={\mathcal K}^{-1} e^{-t u^2} {\mathcal K} ,\;\; t\geq 0,
\end{equation}
which, at first, act on $L_2(\r,\d x)$.
 Here we understand $e^{-t u^2}$ to be the multiplication operator on $L_2((0,\infty), \mu(\d u))$ that maps a function $g(u)$ to $e^{-t u^2} g(u)$.
Taking into account \eqref{K-L+} and \eqref{K-inv}, it is apparent that $\widetilde{\pp_t}$ for $t>0$  is an integral operator with  the  Yakubovich  heat kernel \eqref{p_t}.  Sousa and Yakubovich
\cite[Example 3.7]{SousaYakubovich2018} relate this kernel to the Sturm--Liouville theory of differential operators, so the operators $\widetilde \pp_t$ as integral operators with kernel \eqref{p_t} act also on bounded continuous functions, and define transition probabilities of a sub-Markov process.
The following  is a restatement of these facts.
\begin{theorem}\label{T-YS}
  The Yakubovich heat kernel $p_{t}(x,y)$ defined in \eqref{p_t} has the following properties.
   \begin{enumerate}
     [(i)]
     \item   For $t>0$ and $x,y\in\r$, the kernel is  symmetric, $p_t(x,y)=p_t(y,x)$ and positive, $p_t(x,y)>0$.
     \item For $t>0$ and $x\in\r$, $p_t(x,y)$ defines a sub-probability density function,  $\int_\r p_t(x,y)\d y <1$.
     \item For $s,t>0$ and $x,z\in \r$, the Markov property holds, $\int_\r p_t(x,y)p_s(y,z)\d y=p_{s+t}(x,z)$.
   \end{enumerate}
\end{theorem}
Since we could not locate an appropriate reference for Theorem \ref{T-YS}, we sketch a probabilistic justification
by relating the semigroup $\widetilde \pp_t$ in \eqref{PP}   to the killed Brownian motion.
Let $(\widehat X_t)_{t\geq 0}$ be the process defined by the Markov generator
\begin{equation*}
{\mathcal L}^{\widehat X}f(x)=f''(x)-e^{2x} f(x).
\end{equation*}
We can identify $\widehat X$ with the   $\sqrt{2}$ multiple of Brownian motion that is killed at rate $k(x)=e^{2x}$. More precisely, let $B$ be a standard Brownian motion, then the semigroup of the process $\widehat X$ is given by
\begin{equation}\label{P_t-semi}
\widehat {\mathcal {P}}_t f(x)=\e_x[f(\widehat X_t)]=\e\Big[ e^{-\int_0^t e^{2 \sqrt{2} B_s} \d s} f(\sqrt{2} B_t) \Big\vert  B_0=\frac{x}{\sqrt{2}} \Big].
\end{equation}
It is then clear that $\widehat {\mathcal {P}}_t$  is positive and sub-Markovian, $\widehat \pp_t 1 <1$ for $t>0$.

Here is how one can see  that the semigroup defined by \eqref{P_t-semi} is the same semigroup that we defined in \eqref{PP}.
The modified Bessel function $K_{\i u}(x)$ satisfies the differential equation
$$
x^2 f''(x)+xf'(x)-x^2 f(x)=-u^2 f(x),
$$
thus the function $f(x)=K_{\i u}(e^x)$ satisfies the differential equation
$$
f''(x)-e^{2x} f(x)=-u^2 f(x).
$$
This shows that functions $f_{u}(x)=K_{\i u}(e^x)$ are eigenfunctions of the Markov generator ${\mathcal L}^{\widehat{X}}$ and of the Markov semigroup $\widehat{\pp}_t $:
\begin{equation*}
{\mathcal L}^{\widehat X} f_{u}(x)=-u^2 f_{u}(x), \;\;\; \widehat{\pp }_t f_{u}(x) =f_{u}(x) e^{-t u^2} .
\end{equation*}
We conclude that the integral operator  ${\mathcal K}$ (our modified version \eqref{K-L+} of the Kontorovich--Lebedev transform) diagonalizes the Markov semigroup $\widehat \pp_t$, thus extending the action of semigroup
 \eqref{PP} beyond the initial domain $L_2(\r,\d x)$. We therefore  identify $\widetilde \pp _t$ and $\widehat \pp _t$.
  (We will apply $\widetilde \pp _t$ to functions    of the form $e^{ax}$ which are not in $L_2(\r,\d x)$.)

\begin{remark}
We note that kernel $p_t(x,y)$  appears in \cite[Section 4]{MatsumotoYor2005I} as  $q_1(2t,x,y)$ in their notation.
\end{remark}

\subsection{The normalizing constants}
In view of Theorem  \ref{T-YS}(iii), it is clear that the normalizing constant for the density \eqref{joint-density} is given by the bivariate integral  \eqref{C-formula} which is symmetric in $\vA,\uC$.  We need to show that this integral is finite, and that it is given by an expression that we will use in the proof of Theorem \ref{T1.3}. %
We recall the standard notation for the  Pochhammer symbols and their products:
$$ (a)_n=a(a+1)\dots(a+n-1), \quad (a_1,a_2,\dots,a_k)_n=(a_1)_n(a_2)_n\cdots (a_k)_n.$$

 \begin{theorem}\label{Thm-L-Const}
   If $\vA+\uC>0$ and $\tau>0$ then the integral \eqref{C-formula}
    is finite and %

\begin{equation}
  \label{C2K}
   C\topp \tau_{\vA,\uC}= 2^{\vA+\uC}\left(\CabLa+ \max\{\DabLa,\DbaLa\}\right),
\end{equation}
 with
   \begin{eqnarray}
    \label{C-normalize}
    \CabLa &=& \frac{1}{8\pi}\int_0^\infty e^{-\tau u^2} \frac{\vert \Gamma(\tfrac{\vA+\i u}{2},\tfrac{\uC+\i u}{2})\vert ^2}{\vert \Gamma(\i u)\vert ^2}\d u, \\
    \nonumber \\
   \DabLa &=&
       \frac{\Gamma(\frac{\uC+\vA}{2},\frac{\uC-\vA}{2})}{2\vA \Gamma(-\vA )}\sum_{\{k\geq 0:\; \vA+2k<0 \}} e^{\tau(\vA+2k)^2}(\vA+2k)\frac{(\vA,\frac{\vA+\uC}{2})_k}{k!(1+\frac{\vA-\uC}{2})_k}.
  \label{D-normalize}
\end{eqnarray}
 \end{theorem}
(Note that $ \DabLa=0$ if $\vA\geq 0$, i.e., at most one of $\DabLa$ and $\DbaLa$ is non-zero. In particular, if $\vA,\uC\ge 0$, then $C\topp \tau_{\vA,\uC}= 2^{\vA+\uC}\CabLa$.)

 The proof consists of the following steps: we begin by proving \eqref{C2K}  for $\vA,\uC>0$, then we show that the integral \eqref{C-formula} is finite for all $\vA+\uC>0$, and finally prove that \eqref{C2K} holds also for $\vA\leq 0$.

{
\subsubsection{Proof for the Case $\vA,\uC>0$}
 We use \eqref{p_t} to write the left hand side of \eqref{C-formula} as an iterated  integral
  $$
  C\topp \tau_{\vA,\uC}=\int_{\r^2}  e^{\vA x+ \uC y}  \frac{2}{\pi}\left(\int_0^{\infty} e^{-\tau u^2} K_{\i u}(e^x) K_{\i u}(e^y) \frac{\d u}{{\vert \Gamma(\i u)\vert ^2}} \right)\d x \d y.
 $$
  Since $\vert  K_{\i u}(e^x)\vert \leq K_0(e^x)$ and $\vA,\uC>0$,  the integrand is integrable, see (\ref{K0-bd1}-\ref{K0-bd2}), and we can use Fubini's theorem to change the order of the integration.
  From  the integral identity \eqref{K-Mellin2}
 we see that  $C\topp \tau_{\vA,\uC}$ is  equal to
\begin{equation}
  \label{C-X}
 \frac{2^{\vA+\uC}}{8\pi}   \int_0^\infty e^{-\tau u^2}\frac{\vert \Gamma(\tfrac{\vA+\i u}{2},\tfrac{\uC+\i u}{2})\vert ^2}{\vert \Gamma(\i u)\vert ^2}\d u,
\end{equation}
 which  is  %
 \eqref{C2K} when $\vA,\uC>0$.
}
\subsubsection{Proof of finiteness for  all $\vA+\uC>0$}

Let $\tau>0$ be a fixed constant. We define a function of two variables
\begin{equation}\label{eqn1}
f(\vA,\uC)= \frac{1}{2\pi \i} \int_{\i \r} e^{\tau z^2} \frac{\Gamma((\vA+ z)/2,(\vA- z)/2,(\uC+z)/2,(\uC- z)/2)}
{\Gamma(z) \Gamma(-z)} \d z.
\end{equation}
The integral is taken over the imaginary line in the complex plane. Writing $z=\i u$ we obtain an equivalent representation
\begin{equation*}\label{eqn1b}
f(\vA,\uC)= \frac{1}{2\pi} \int_{\r} e^{-\tau u^2} \frac{\Gamma((\vA+ \i u)/2,(\vA- \i u)/2,(\uC+\i u)/2,(\uC- \i u)/2)}
{\vert \Gamma(\i u)\vert ^2} \d u.
\end{equation*}
Note that the integrand in \eqref{eqn1} (as a function of $z$) is meromorphic and has poles at points
$$
\{\pm (\vA+2n), \pm (\uC+2n) \; : \; n=0,1,2,3,\dots\},
$$
and it is clear that $f(\vA,\uC)$ is an analytic function of two variables $(\vA,\uC)$ in the domain
$$
D_0:=\{(\vA,\uC) \in \c^2 \; : \; \re(\vA)>0, \re(\uC)>0\}.
$$

\begin{theorem}\label{thm1}
The function $f(\vA,\uC)$ can be analytically continued to a function holomorphic in
$$
\Omega=\{(\vA,\uC) \in \c^2 \; : \; \re(\vA+\uC)>0\}.
$$
\end{theorem}

\begin{proof} Without loss of generality we assume that $\re(\uC)>0$. With $\uC$ fixed, consider $\vA$ in the strip $0<\re(\vA)<\re(\uC)$.
Take any number $c_1$ from the interval
$(\re(\vA),\min(\re(\uC),\re(\vA)+2))$. We shift the contour of integration in
\eqref{eqn1} from $\i \r$ to $c_1+\i \r$ and collect the residue at $z=\vA$ (where the integrand has a simple pole) to obtain the following expression
\begin{align}\label{eqn2}
f(\vA,\uC)&=\frac{1}{2\pi \i} \int_{c_1+\i \r} e^{\tau z^2} \frac{\Gamma((\vA+ z)/2,(\vA- z)/2,(\uC+z)/2,(\uC- z)/2)}
{\Gamma(z)\Gamma(-z)} \d z\\ \nonumber
&\quad +2  e^{\tau \vA^2}\frac{\Gamma((\uC+\vA)/2,(\uC- \vA)/2)}{\Gamma(-\vA)}.
\end{align}

The poles at $\vA, \vA+2$ and  $-\vA$
should not be %
lying on the contour of integration
$c_1+i\r$. This gives us conditions
$\re(\vA)<c_1<\re(\vA)+2$, $-\re(\vA)<c_1<\re(\uC)$, in particular $-\re(\vA)<\re(\vA)+2$, or $-1<\re(\vA)$. The second term  in \eqref{eqn2} is analytic if $\re(\uC \pm \vA)>0$.
We conclude that \eqref{eqn2} gives an analytic continuation of $f(\vA,\uC)$ into the following domain:
$$
\widetilde D_1:=\{(\vA,\uC)\in \c^2 \; : \; -1<\re(\vA)<\re(\uC), \re(\vA)+\re(\uC)>0\}.
$$
Now take $(\vA,\uC)\in \widetilde D_1$ with $\re(\vA)<0$. We shift the contour of integration from $c_1+\i \r$ back to $\i \r$, but now we need to take into account the pole at $z=-\vA$:
\begin{align}\label{eqn2b}
f(\vA,\uC)&=\frac{1}{2\pi \i} \int_{\i \r} e^{\tau z^2} \frac{\Gamma((\vA+ z)/2,(\vA- z)/2,(\uC+z)/2,(\uC- z)/2)}
{\Gamma(z)\Gamma(-z)} \d z\\ \nonumber
&\quad +4  e^{\tau \vA^2}\frac{\Gamma((\uC+\vA)/2,(\uC- \vA)/2)}{\Gamma(-\vA)}.
\end{align}
Again, we cannot   have poles at   $\vA,\vA+2,-\vA-2,-\vA$ lying on
the contour of integration $i\r$, thus we have conditions
$-2< \re(\vA) < 0$ and $\re(\uC)>0$.
Again, the second term in \eqref{eqn2b} is analytic as long as $\re(\uC\pm \vA)>0$. Therefore,
formula \eqref{eqn2b} gives us an analytic continuation of $f(\vA,\uC)$ in the domain
$$
D_2:=\{(\vA,\uC)\in \c^2 \; : \; -2<\re(\vA)<0, \re(\vA)+\re(\uC)>0\}.
$$

Now we repeat this procedure. Take $(\vA,\uC)\in D_2$ such that $-2<\re(\vA)<-1$ and $\re(\vA)+2<\re(\uC)$. Choose any $c_2$ in the interval $(\re(\vA)+2,\min(\re(\uC),\re(\vA)+4))$. Shift the contour of integration in \eqref{eqn2b} from $\i \r$ to $c_2+\i \r$,  and collect the residue at $z=\vA+2$ and we get
\begin{align*}
f(\vA,\uC)&=\frac{1}{2\pi \i} \int_{c_2+\i \r} e^{\tau z^2} \frac{\Gamma((\vA+ z)/2,(\vA- z)/2,(\uC+z)/2,(\uC- z)/2)}
{\Gamma(z)\Gamma(-z)} \d z\\ \nonumber
&\quad +
4  e^{\tau \vA^2} \frac{\Gamma((\uC+\vA)/2,(\uC- \vA)/2)}{\Gamma(-\vA)} \\ \nonumber
&\quad -  2 e^{\tau(\vA+2)^2} \frac{\Gamma(\vA+1,(\uC+\vA+2)/2,(\uC- \vA-2)/2)}{\Gamma(\vA+2)\Gamma(-\vA-2)}.
\end{align*}
We cannot have  poles at $\vA+2$, $\vA+4$, $-\vA-2$, $-\vA$
lying on  the contour of integration $c_2+i\r$, so
the above expression gives us analytic continuation of $f(\vA,\uC)$ into the following domain:
$$
D_3:=\{(\vA,\uC)\in \c^2 \; : \; -3<\re(\vA)<-1, \re(\vA)+\re(\uC)>0\}.
$$
Now take $(\vA,\uC)\in D_3$ such that $-3<\re(\vA)<-2$. We shift the contour of integration from $c_2+\i \r$ back to $\i \r$, but now we need to take into account the pole at $z=-\vA-2$:
\begin{align*}
f(\vA,\uC)&=\frac{1}{2\pi \i} \int_{\i \r} e^{\tau z^2} \frac{\Gamma((\vA+ z)/2,(\vA- z)/2,(\uC+z)/2,(\uC- z)/2)}
{\Gamma(z)\Gamma(-z)} \d z\\ \nonumber
&\quad +
4  e^{\tau \vA^2} \frac{\Gamma((\uC+\vA)/2,(\uC- \vA)/2)}{\Gamma(-\vA)} \\ \nonumber
&\quad -  4 e^{\tau(\vA+2)^2} \frac{\Gamma(\vA+1,(\uC+\vA+2)/2,(\uC- \vA-2)/2)}{\Gamma(\vA+2)\Gamma(-\vA-2)}.
\end{align*}
Again, the poles at $\vA+2,\vA+4,-\vA-4,-\vA-2 $ should not be %
lying on
the contour of integration   $i\r$. So
the above expression gives us analytic continuation of $f(\vA,\uC)$ into the following domain:
$$
D_4:=\{(\vA,\uC)\in \c^2 \; : \; -4<\re(\vA)<-2, \re(\vA)+\re(\uC)>0\}.
$$

After repeating the above procedure, the function $f(\vA,\uC)$ would be analytically continued into domains $D_5$, $D_6$, etc., where we defined
$$
D_n:=\{(\vA,\uC)\in \c^2 \; : \; -n<\re(\vA)<-n+2, \re(\vA)+\re(\uC)>0\}.
$$
Note
$$\widetilde D_1 \cup D_2 \cup D_3 \cup D_4 \cup \dots =
\{(\vA,\uC) \in \c^2 \; : \; \re(\vA)+\re(\uC)>0, \re(\uC)-\re(\vA)>0\},
$$
thus the function $f(\vA,\uC)$ is analytic in the above domain. Since $f(\vA,\uC)=f(\uC,\vA)$
and this function is analytic in
$$
\{(\vA,\uC) \in \c^2 \; : \; \re(\vA)>0, \re(\uC)>0\},
$$
we conclude that $f(\vA,\uC)$ can be analytically continued to a function holomorphic in
$\Omega$.
\end{proof}

The finiteness part of the conclusion of Theorem \ref{Thm-L-Const} now follows.
\begin{corollary}\label{cor:fini}
Let $\tau>0$.
The integral \eqref{C-formula}
is finite for real $\vA,\uC$ such that $\vA+\uC>0$.
\end{corollary}
\begin{proof}
 Recall that the condition $C_{\vA,\uC}\topp \tau<\infty$ for all real $-\uC<\vA<\infty$ is equivalent to
the statement that the Laplace transform $\vA\mapsto C_{\vA,\uC}\topp \tau$ has analytic extension to the complex half-plane  $\re(\vA)>-\uC$. (For a
version of this fact in the language of analytic characteristic functions, see
\cite[Theorem 2]{lukacs1952}.)  Since $C\topp \tau_{\vA,\uC}=2^{\vA+\uC-3} f(\vA,\uC)$,
 Theorem \ref{thm1} ends the proof.
\end{proof}
\begin{remark}
Corollary \ref{cor:fini} can also be deduced from the fact that the integral \eqref{C-formula}
is finite in the upper quadrant $\vA,\uC>0$ by \eqref{C-X},  and in the strip $0<\vA+\uC<2$ by Theorem \ref{Prop-LapC}.  So it is finite in the convex hull of the union of these two sets.
\end{remark}
\subsubsection{Proof of formula \eqref{C2K}}
We now consider  $\vA\in\r,\uC>0$ such that $\vA+\uC>0$.
The gamma function $\Gamma(z)$ has simple poles at points $z=-k$, $k=0,1,2,\dots$ with residues at these points given by
$\frac{(-1)^k}{k!}$. Thus with $\vA+2k<0$, and $\vA+k\ne 0, -1,-2,\dots$  the residue of the integrand in \eqref{eqn1} at point $z=\vA+2k$ is given by
\begin{align*}
-2 \frac{(-1)^k}{k!} e^{\tau (\vA+2k)^2} \frac{\Gamma(\vA+k) \Gamma((\uC+\vA)/2+k) \Gamma((\uC-\vA)/2-k)}{\Gamma(\vA+2k) \Gamma(-\vA-2k)}.
\end{align*}
This expression can be simplified to
\begin{align*}
-2 \frac{\Gamma((\uC+\vA)/2,(\uC-\vA)/2)}{\vA\Gamma(-\vA)} e^{\tau(\vA+2k)^2}(\vA+2k)\frac{(\vA,(\uC+\vA)/2)_k}{k!(1+(\vA-\uC)/2)_k},
\end{align*}
removing the singularity at $\vA+k=0,-1,-2,\dots$.
\arxiv{One can use here Euler's reflection formula $\Gamma(z)\Gamma(1-z)=\pi/\sin(\pi z)$, $z\not\in\mathbb{Z}$ and
$
\frac{\Gamma(a+n)}{\Gamma(a)}=(a)_n, \; a\not\in\{\dots,-1,0\}$,
$
\frac{\Gamma(a-n)}{\Gamma(a)}=\frac{(-1)^n}{(1-a)_n}, \;   a\not\in\{\dots,-1, 0,1,\dots,n\},\; n=0,1,\dots
$}

Similarly, the residue at $z=-\vA-2k$ is given by
\begin{align*}
2 \frac{\Gamma((\uC+\vA)/2,(\uC-\vA)/2)}{\vA\Gamma(-\vA)} e^{\tau(\vA+2k)^2}(\vA+2k)\frac{(\vA,(\uC+\vA)/2)_k}{k!(1+(\vA-\uC)/2)_k}.
\end{align*}
Thus if  $-2n<\vA<-2(n-1)$ and $\vA+\uC>0$, we have
\begin{align*}
f(\vA,\uC)&=\frac{1}{2\pi \i} \int_{\i \r} e^{t z^2} \frac{\Gamma((\vA+ z)/2,(\vA- z)/2,(\uC+z)/2,(\uC- z)/2)}
{\Gamma(z)\Gamma(-z)} \d z\\ \nonumber
&\quad+ 4 \frac{\Gamma((\uC+\vA)/2,(\uC-\vA)/2)}{\vA\Gamma(-\vA)}
\sum\limits_{k=0}^{n-1} e^{t(\vA+2k)^2}(\vA+2k)\frac{(\vA,(\uC+\vA)/2)_k}{k!(1+(\vA-\uC)/2)_k}.
\end{align*}
Let us express this identity in a different way. We change variable $z=\i u$ in the integral, use the symmetry of the integrand with respect to $u \mapsto -u$ and obtain
\begin{align}\label{eqn4.5}
f(\vA,\uC)&=\frac{1}{\pi} \int_{0}^{\infty} e^{-\tau u^2} \frac{\Gamma((\vA+ \i u)/2,(\vA- \i u)/2,(\uC+\i u)/2,(\uC- \i u)/2)}
{\Gamma(\i u)\Gamma(-\i u)} \d u\\ \nonumber
&\quad + 4 \frac{\Gamma((\uC+\vA)/2,(\uC-\vA)/2)}{\vA\Gamma(-\vA)}
\sum\limits_{k=0}^{n-1} e^{\tau(\vA+2k)^2}(\vA+2k)\frac{(\vA,(\uC+\vA)/2)_k}{k!(1+(\vA-\uC)/2)_k}.
\end{align}
Thus for real $-2n<\vA<-2(n-1)$ and $\vA+\uC>0$ we have
\begin{multline*}
C\topp \tau_{\vA,\uC}=2^{\vA+\uC-3}f(\vA,\uC) \\=\frac{2^{\vA+\uC-3}}{\pi} \int_{0}^{\infty} e^{-t u^2} \frac{\Gamma((\vA+ \i u)/2,(\vA- \i u)/2,(\uC+\i u)/2,(\uC- \i u)/2)}
{\Gamma(\i u)\Gamma(-\i u)} \d u\\ \nonumber
 \quad + 2^{\vA+\uC-1} \frac{\Gamma((\uC+\vA)/2,(\uC-\vA)/2)}{\vA\Gamma(-\vA)}
\sum\limits_{k=0}^{n-1} e^{t(\vA+2k)^2}(\vA+2k)\frac{(\vA,(\uC+\vA)/2)_k}{k!(1+(\vA-\uC)/2)_k}
\\
=2^{\vA+\uC}(\CabLa+\DabLa),
\end{multline*}
proving \eqref{C2K} for all $\vA\ne 0,-2,-4,\dots$.
It remains to note that by continuity of  finite Laplace transforms, function $\vA\to C\topp \tau_{\vA,\uC}$ is continuous. By the dominated convergence theorem,  we can pass to the limits $a\to-2k$, $k=0,1,\ldots$  under the integral sign   in %
\eqref{eqn4.5}.    Thus \eqref{C2K} holds for all $\vA+\uC>0$.

\section{Proof of Theorem \ref{T1.1}}
\label{Sect:Recipr}
Recall semigroup \eqref{PP} and   (symmetric) Yakubovich heat  kernel \eqref{p_t}.
From the formula for the joint density \eqref{joint-density} it is clear that time reversal   $X_t=Y_{\tau-t}$  preserves the form of the density but swaps the roles of parameters $\vA,\uC$.
We  give the proof for the case $\uC>0$ and construct Markov process $(X_t)$ which will give process $(Y_t)$ by time reversal. (Otherwise, we would use $\vA>0$ in the construction and construct process $(Y_t)$ directly.)

 From the semi-group property of  $\widetilde {\pp }_t$, we see that
\begin{equation*}\label{H}
  H_t(x):= \int_\r p_{\tau-t}(x,y)e^{\uC y}\d y, \quad 0\leq t<\tau,
\end{equation*}
is a positive  space-time harmonic    function, so we can use it as Doob's $h$-transform. We remark that identity  \eqref{K-Mellin2} shows that
 \begin{equation*}
 \label{eq:H_t}
  H_t(x)=
 2^{\uC-2}\int_0^\infty K_{\i u}(e^x)e^{-u^2(\tau-t)}\left\vert \Gamma\left(\frac{\uC+\i u}{2}\right)\right\vert ^2\mu(\d u),\; t<\tau,
 \end{equation*}
with $\mu$ given by \eqref{mu}.
 Additionally, set
$H_\tau(x):=e^{\uC x}$. Then
  for $0\leq s<t\leq \tau$ we define probability measures
\begin{equation}\label{P_st}
  P_{s,t}(x,dy)=
     \frac{H_t(y)}{H_s(x)} p_{t-s}(x,y)\d y.
\end{equation}
 It is clear that $P_{s,t}(x,\d y)$ are transition probabilities of a Markov  %
process
 $(X_t)_{0\leq t\leq \tau}$ with state space $\r$.

 The second real parameter, $\vA$,   enters the initial distribution as follows. By Theorem \ref{Thm-L-Const} and Fubini's theorem,
function  $H_0(x) e^{\vA x}$ is integrable, with the integral
$$\int_\r H_0(x) e^{\vA x}\d x=\int_\r\int_\r e^{\vA x+\uC y} p_{\tau}(x,y)\d x \d y =C\topp \tau_{\vA,\uC}.$$
We normalize this function and take it as  the  initial distribution for $X_0$,
\begin{equation}\label{X-ini}
  \p(X_0\in \d x)=\frac{1}{C\topp \tau_{\vA,\uC}}H_0(x) e^{\vA x}\d x.
\end{equation}
This completes the construction of the Markov process  $(X_t)_{0\leq t\leq \tau}$.

 From   \eqref{P_st} and \eqref{X-ini}, it is clear that  the joint law of $(X_0,X_\tau)$ is  $$\frac{1}{C\topp \tau_{\vA,\uC}}e^{\vA x} p_\tau(x,y)e^{\uC y}\d x \d y.$$
More generally, for $0=t_0<t_1<\dots<t_{d}<t_{d+1}=\tau$, the joint distribution of $(X_{t_0},\dots,X_{t_{d+1}})$ is
\begin{equation}\label{joint-law0}
   \frac{1}{C\topp \tau_{\vA,\uC}} e^{\vA x_0} \d x_0\prod_{j=1}^{d+1} p_{t_j-t_{j-1}}(x_{j-1},x_j) \d x_1\dots \d x_{d}e^{\uC x_{d+1}} \d x_{d+1}.
\end{equation}
This is \eqref{joint-density} except that parameters $\vA,\uC$ are at ``incorrect locations''. We swap them by time-reversal.
A calculation shows that  the finite dimensional distributions of process $Y_t=X_{\tau-t}$ have density \eqref{joint-density}. This proves Theorem \ref{T1.1}. \qed
\arxiv{
To verify that the finite dimensional distributions for  $Y_t=X_{\tau-t}$ are given by density \eqref{joint-density}, fix $d\geq 0$ and
 $t_0<t_1<\dots<t_{d+1}=1$. Let $\bar t_j=1-t_{d+1-j}$ and $\bar x_j=x_{d+1-j}$, $0\leq j\leq d+1$.  We have
  $0=\bar t_0<\bar t_1<\dots<\bar t_d<\bar t_{d+1}=1$
  and the joint density of  $(Y_{t_0},\dots,Y_{t_{d+1}})$ at $\vv x=(x_0,\dots,x_{d+1})$ is the joint density of
  $(X_{\bar t_0},\dots,X_{\bar t_{d+1}})$
 at   $\bar{\vv x}=(\bar x_j)$. We apply \eqref{joint-law0} with  $\{\bar t_j\}$ and $\{\bar x_j\}$.  Using symmetry $p_t(x,y)=p_t(y,x)$, we  get
\begin{align*}
    \frac{1}{C\topp \tau_{\vA,\uC}} e^{\vA \bar x_0} \d \bar x_0 & \prod_{j=1}^{d+1} p_{\bar t_j-\bar t_{j-1}}(\bar x_{j-1},\bar x_j) \d \bar x_1\dots \d \bar x_{d}e^{\uC \bar x_{d+1}} \d \bar x_{d+1}
 \\
& = \frac{1}{C\topp \tau_{\vA,\uC}} e^{\vA   x_{d+1}} \d  x_{d+1} \prod_{j=1}^{d+1} p_{t_{d+2-j}- t_{d+1-j}}( x_{d+2-j}, x_{d+1-j}) \d  x_1\dots \d  x_{d}e^{\uC  x_{0}} \d  x_{0}
  \\
& = \frac{1}{C\topp \tau_{\vA,\uC}} e^{\vA   x_{d+1}} \d  x_{d+1} \prod_{j=1}^{d+1} p_{t_{j}- t_{j-1}}( x_{j}, x_{j-1}) \d  x_1\dots \d  x_{d}e^{\uC  x_{0}} \d  x_{0}
  \\
& = \frac{1}{C\topp \tau_{\vA,\uC}} e^{\vA   x_{d+1}} \d  x_{d+1} \prod_{j=1}^{d+1} p_{t_{j}- t_{j-1}}( x_{j-1}, x_{j}) \d  x_1\dots \d  x_{d}e^{\uC x_{0}} \d  x_{0},
\end{align*}
which is \eqref{joint-density}.
}

 We remark that the construction of process $(X_t)$ is a special case of  the standard construction of a  {\em reciprocal Markov process} (also known as the Schr\"odinger or Bernstein processes \cite{jamison1974reciprocal,jamison1975markov}) corresponding to a sub-Markovian semigroup. A concise and readable overview of this construction %
 with additional references
 appears in the first two sections of \cite{dawson1990schrodinger}.

\section{ Proof of  Theorem \ref{T1.3}}\label{Sec:ProofT1.3}
 The first step is to relate the  normalizing constant $\Kab$ from \eqref{Kab2p0} to the normalizing constant $C_{\vA,\uC}\topp \tau$ from \eqref{C-formula}.

 \begin{proposition} If $\vA+\uC>0$, then the constants \eqref{Kab2p0} %
 and \eqref{C-formula} are related as follows
    \begin{equation}\label{K2K}
   \Kab %
   =\frac{(\vA+\uC)(\vA+\uC+2)}{2^{\vA+\uC+1}}C_{\vA,\uC}\topp \tau. %
 \end{equation}

 \end{proposition}
 \begin{proof} By symmetry, without loss of generality we assume $\uC>0$.
  From \eqref{their-nu} we see that the normalizing constant \eqref{Kab2p0}, compare \cite[(1.12)]{CorwinKnizel2021}, is
  \begin{align*}
   \Kab& =  \int_\r e^{-\tau x} \mathfrak p_0(\d x)\\
 &  =
   \frac{(\vA+\uC)(\vA+\uC+2)}{16\pi}\int_0^\infty e^{-\tau x}  \frac{\vert \Gamma(\frac{\vA+\i \sqrt{x}}{2},\frac{\uC+\i \sqrt{x}}{2} )\vert ^2}{2\sqrt{x}\vert \Gamma(\i \sqrt{x})\vert ^2}\d x
     \\&+    \sum_{\{j\ge 0:\; 2j+\vA<0\}}e^{\tau(\vA+2j)^2} p_j(0)\nonumber
   \\
&
  =    \frac{(\vA+\uC)(\vA+\uC+2)}{16\pi}\int_0^\infty e^{-\tau x}  \frac{\vert \Gamma(\frac{\vA+\i \sqrt{x}}{2},\frac{\uC+\i \sqrt{x}}{2} )\vert ^2}{2\sqrt{x}\vert \Gamma(\i \sqrt{x})\vert ^2}\d x\nonumber\\
& \quad   +     \frac{(\vA+\uC)(\vA+\uC+2)}{4}  \frac{\Gamma(\frac{\uC-\vA}{2},\frac{\vA+\uC}{2})}{\vA\Gamma(-\vA)} \sum_{\{j\ge 0:\; 2j+\vA<0\}} e^{\tau(\vA+2j)^2} \frac{(\vA+2j)(\vA,\frac{\vA+\uC}{2})_j}{ j!(1+\frac{\vA-\uC}{2})_j}.\nonumber
  \end{align*}
  Substituting a new variable for $\sqrt{x}$ in the integral and invoking formulas \eqref{C2K}-\eqref{D-normalize}, we get  \eqref{K2K}.
 \end{proof}

\begin{proof}[Proof of Theorem \ref{T1.3}]
 With $t_d=t_{d+1}=1$ in \eqref{psi},
we have
\begin{align*}
  \psi\topp \tau(\vv s,\vv t)&=\frac{1}{\Kab}\int_\r \e\left[\exp\left(-\tau \sum_{k=1}^{d}(t_k-t_{k-1})\mathbb{T}_{s_k}\right)\middle\vert \mathbb{T}_{0}=x\right]\mathfrak p_{0}(\d x)
\\  =&
  \frac{1}{\Kab}\int_\r \e\left[\exp\left(-\tau \sum_{k=1}^{d}(t_k-t_{k-1})\mathbb{T}_{s_k}\right)\middle\vert \mathbb{T}_{s_{d}}=y\right]\int_\r\mathfrak p_{0,s_d}(x, \d y) \mathfrak p_{0}(\d x).
\end{align*}
 From \eqref{p-inv}, we see that
\begin{equation*}\label{psi+}
  \psi\topp\tau(\vv s,\vv t)=%
  \int_\r \e\left[\exp\left(-\tau \sum_{k=1}^{d}(t_k-t_{k-1})
  \mathbb{T}_{s_k}\right)\middle\vert \mathbb{T}_{s_d}=x\right]
\frac{1}{\Kab}\mathfrak p_{s_d}(\d x). \end{equation*}
Since $s_d>-\vA$, from \eqref{their-nu} we see that measure $\mathfrak p_{s_d}(\d x)$ is absolutely continuous. Furthermore,  comparing \eqref{K2K} with \eqref{their-nu},  we get
\begin{equation*}
  \label{p2phi}
  \frac{1}{\Kab}\mathfrak p_s(\d x)=\widetilde {\mathfrak p}_s(x)\d x, \; -\vA<s<\uC,
\end{equation*}
where $\widetilde {\mathfrak p}_s$ is given by \eqref{tilde-mathfrak-p}.
So the left hand side of  \eqref{psi-dual} is the same as the left hand side of  \eqref{Dual-1}, with  $d-1\geq 0$ instead of $d$.
By the Abel transformation
\begin{equation*}
  \label{Abel}
  \sum\limits_{k=1}^{d}  a_{k} (b_{k}-{b_{k-1}})=\sum\limits_{k=1}^{d-1}  (a_{k}-a_{k+1}) (b_{k}-b_{0})+a_{d}(b_{d}-b_{0}),
\end{equation*}
applied to the right hand side of \eqref{Dual-1}, we get \eqref{psi-dual}.

\end{proof}

\section{The pair of dual   semigroups and  dual representation for the Laplace transforms}
 \label{Sect: DMS}

This section uses the Kontorovich-Lebedev transform $\kk$ to  explain the dual representation for the Laplace transforms in Theorem \ref{T1.2} by relating
 processes $(Z_t)$, introduced in Section \ref{Sect:SecondProof},  to process  $(X_t)$, introduced in Section \ref{Sect:Recipr}. We discuss how the semigroups of these two processes are connected, and then we use  Kontorovich-Lebedev transform to relate their Laplace transforms.
 \subsection{Dual semigroups}
 Recall that in   Section \ref{Sect:Recipr} we used the sub-Markovian semi-group $(\widetilde \pp_t)_{t>0}$ defined by \eqref{PP}
 to construct the Markov
 transition operators
  $\pp_{s,t}=\frac{1}{H_s}\widetilde\pp_{t-s}H_t$, i.e.,
\begin{equation}\label{P-H}
\pp _{s,t}[f]=\frac{1}{H_s}\widetilde {\pp }_{t-s}[H_t  f] , \quad 0\leq s<t\leq \tau,
\end{equation}
with  $H_t=\kk^{-1}\ee^{-u^2(T-t)}\kk e^{\uC x}$.

Now we use the Kontorovich-Lebedev transform  to define another semigroup, acting on  $L_2((0,\infty),\mu(\d u))$ by reversing the %
roles of $\kk^{-1}$ and $\kk$. For $s>0$, let
\begin{equation}\label{tilde Q}
\widetilde {\qq}_s={\mathcal K} e^{sx} {\mathcal K}^{-1},
\end{equation}
where $e^{sx}$ is a multiplication operator on $L_2((0,\infty), \d x)$.
Using the integral identity \eqref{K-Melin}
we find that the operator $\widetilde{\qq}_t$ has kernel
\begin{equation}\label{q_t}
\widetilde q_t(u,v)=  \frac{2^{t}}{4\pi\Gamma(t)\vert \Gamma(\i v)\vert ^2}
\vert \Gamma((t+\i (u+v))/2,(t+\i (u-v))/2)\vert ^2,
\end{equation}
i.e., $\widetilde \qq_t[f](u)=\int \widetilde q_t (u,v) f(v) \d v$. %
Note the difference in notation with  tilde for the kernel  $\widetilde q_t$ of $\widetilde \qq_t$   and no tilde for $q_{s,t}$  in \eqref{ICAK-q}. %

It is clear that   $(\widetilde \qq_s)_{s>0}$ satisfies the semigroup property
$\widetilde {\qq}_{t+s}=\widetilde {\qq}_t \widetilde{\qq}_s$  %
on $L_2(\d \mu)$, and the  kernel is positive. However, there is a major problem in that $\widetilde {\qq}_t 1 = +\infty$ for all $t>0$. In other words, the measures $\widetilde q_t(u,v) \d v$ are not probability measures. So we will need to do some further work to turn $\widetilde {\qq}_t$ into a Markov semigroup.

To define Markov semigroup $\qq_{s,t}$, we  proceed as in \eqref{P-H}. For $s<\uC$ and $u>0$ we introduce the functions
\begin{equation}\label{h_s}
  h_s(u)= \mathcal{K} [e^{(\uC-s) x}](u) = 2^{\uC-s-2} \vert \Gamma(\tfrac{\uC-s+\i u}{2})\vert ^2,
\end{equation}
where we used identity \eqref{K-Mellin2}.
For $s<t<\uC$ we introduce operators
\begin{equation*}\label{Qst0}
{\qq}_{s,t}=\frac{1}{h_{s}} \widetilde {\qq}_{t-s} h_{t}.
\end{equation*}
  Informally,  one can think of ${\qq}_{s,t}$ as Doob's $h$-transform of $\widetilde {\qq}_t$.
 Now it is clear that operators ${\qq}_{s,t}$ have an integral kernel
\begin{align*} \label{QH2}
q_{s,t}(u,v)&=\frac{h_{t}(v)}{h_{s}(u)} \widetilde q_{t-s}(u,v)\\ \nonumber
&=\frac{\vert \Gamma((\uC-t+\i v)/2,(t-s+\i (u+v))/2,(t-s+\i (u-v))/2)\vert ^2}{4\pi\Gamma(t-s)\vert \Gamma((\uC-s+\i u)/2,\i v)\vert ^2} ,
\end{align*}
matching  \eqref{ICAK-q}.

 The following proposition  says that  $q_{s,t}(u,v)$ is a kernel of a Markov semigroup and  it  is known; we   include its proof  in  Appendix \ref{Sect:SemiQ} for completeness.
\begin{proposition}\label{Prop-Semi}
 ${\qq}_{s,t} 1=1$, and for $t_1<t_2<t_3<\uC$ we have
$$
{\qq}_{t_1,t_2} {\qq}_{t_2,t_3} = {\qq}_{t_1,t_3}.
$$
\end{proposition}
So operators $\qq_{s,t}$ are well defined  Markov transition operators for all real $-\infty<s<t<\uC$.
To define the ``initial distributions" of Markov process $(Z_s)$, for  $-\vA<s<\uC$, we   let
 \begin{equation}
   \label{gsu}
   g_s(u) = \kk[ e^{(s+\vA) x}](u) =2^{\vA+s-2} \vert \Gamma(\tfrac{\vA+s+\i u}{2})\vert ^2
 \end{equation}
(recall \eqref{K-Mellin2}) and we introduce the family of $\sigma$-finite positive measures
\begin{equation}
  \label{Z-ini}
\widetilde\nu_s(\d u)=\frac{1}{C_{\vA,\uC}\topp \tau}h_{s}(u)g_{s}(u) \mu(\d u)
=
\frac{2^{\vA+\uC}}{8\pi C_{\vA,\uC}\topp\tau}
 \frac{\vert \Gamma(\tfrac{\vA+s+\i u}{2},\tfrac{\uC-s+\i u}{2})\vert ^2}{\vert \Gamma(\i u)\vert ^2}\d u = \varphi_s(u)\d u.
\end{equation}
The   normalization of the  infinite measure $\widetilde\nu_s(\d u)$ is chosen to match \eqref{X-ini},  %
explaining the normalization in
\eqref{nu_s}.

It is known that measures $\widetilde \nu_s$ are %
 the entrance laws for $\qq_{s,t}$,
compare \cite[Lemma 7.11]{CorwinKnizel2021} and \eqref{p-inv}.
\begin{lemma}\label{Rem:Yizao}  For  $-\vA<s<t<\uC$
$$
 \widetilde \nu_s\qq_{s,t}=\widetilde \nu_t.
$$
Equivalently,
  $$\int_0^\infty q_{s,t}(u,v)g_s(u)h_s(u)\mu(\d u) \d v=g_t(v)h_t(v)\mu(\d v).$$
\end{lemma}
\begin{proof}
   We give a proof for completeness.
Using the symmetry of \eqref{q_t} with respect to $u,v$ and explicit expression \eqref{mu}, the density on left hand side is
\begin{align*}
 \int_0^\infty h_s(u)q_{s,t}(u,v)g_s(u)\mu(\d u)&  =
  \frac{h_t(v)}{\vert \Gamma(\i v)\vert ^2} \int_0^\infty \widetilde \vert \Gamma(\i v)\vert ^2 \widetilde q_{t-s}(u,v)g_s(u)\mu(\d u)
\\
&=\frac{2}{\pi} \frac{h_t(v)g_t(v)} {\vert \Gamma(\i v)\vert ^2}  \int_0^\infty  \frac{1}{g_t(v)}\frac{\vert \Gamma(\i v)\vert ^2\widetilde q_{t-s}(u,v)}{\vert \Gamma(\i u)\vert ^2}g_s(u)\d u
\\
&=\frac{2}{\pi} \frac{h_t(v)g_t(v)}{\vert \Gamma(\i v)\vert ^2} \int_0^\infty  \frac{1}{g_t(v)}\widetilde q_{t-s}(v,u)  g_s(u) \d u.
\end{align*}
Formal calculations indicate that the last integral should be one, as it is equal to
$$\frac{1}{g_t}\left(\kk e^{(t-s)x}\kk^{-1}\right)\kk [e^{(s+\vA)x}].$$
To avoid justification of the associative law in this setting,  we just write down the integral explicitly
$$
\int_0^\infty \frac{\vert \Gamma(\frac{\vA+s+\i u}{2})\vert ^2\vert \Gamma((t-s+\i (u+v))/2,(t-s+\i (u-v))/2)\vert ^2}{4\pi \Gamma(t-s)\vert \Gamma(\frac{\vA+t+\i v}{2},\i u)\vert ^2} \d u,
$$
and note that its value is indeed 1, as this is the integral of the kernel \eqref{ICAK-q} (in a different parametrization).
\end{proof}

\label{Sec:Z}
 To summarize, in this section we re-introduced Markov
 processes
$(Z_s)_{-\vA<s<\uC}$ with state space $(0,\infty)$. For $-\infty<s<t<\uC$, the transition probabilities $q_{s,t}(u,v)dv$ depend on parameter $\uC$ and are given
by \eqref{ICAK-q}. The $\sigma$-finite measures $\widetilde\nu_{s}(\d u)=\varphi_{s}(u)\d u$ for $s\in(-\vA,\uC)$ are the entrance laws for this %
 family of transition probabilities.

\subsection{Dual representations for the Laplace transforms}\label{Sect:DualReps}
As we mentioned in  the introduction, the goal of   a dual representation   is to exchange the arguments of the multipoint Laplace transform and the time arguments of the processes.   Our first  %
  dual representation
is conditional, so it does not rely on the choice of the initial distributions for the processes $(X_t)$ and $(Z_s)$. It serves as a lemma for the second %
  dual representation, but it is interesting in its own right.

\begin{theorem}\label{T1}  Let $(X_t)_{t\in[0,\tau]}$ be the Markov process
with finite dimensional distributions \eqref{joint-law0} and let  $(Z_s)_{s\in(-\infty, \uC)}$ be the Markov %
 process
described %
 at the end of Section  \ref{Sec:Z}.
Let $-\infty<s_0<s_1<\dots<s_{n}<\uC$   and $0=t_0<t_1<t_2<\dots<t_{n}<t_{n+1}=\tau$.
 For $x\in\r$ define
\begin{equation}\label{F-W}
  \mathsf  F  (x)=    \e \left[\exp\left(-\sum\limits_{k=0}^n s_{k}(X_{t_{k+1}}-X_{t_k})\right)\middle \vert  X_0=x  \right]
\end{equation}
and for $u>0$ let
\begin{equation}
  \label{Gu}
  \mathsf  G(u)  =\e\left[
    \exp\left(-\sum_{k=1}^n (t_{k+1}-t_k)Z_{s_k}^2\right) \middle \vert {Z_{s_0}=u}  \right].
\end{equation}
Then $ e^{-t_1u^2}h_{s_0}(u) \mathsf G (u)$ is in $L_1(\d \mu)$ and
   \begin{equation}\label{Alexeys-F-G}
e^{-s_0 x} H_{0}(x)  \mathsf  F(x)=\kk^{-1}\left[  e^{-t_1u^2}h_{s_0}(u) \mathsf G(u)\right](x).
\end{equation}
\end{theorem}
\begin{proof}
The proof is based on the identity
\begin{equation}\label{e-interwening}
\kk e^{sx}\widetilde\pp_t  =\widetilde\qq_s e^{-tu^2}\kk,
\end{equation}
where the operators act on $L_1(K_0)$.
This follows from \eqref{PP} and \eqref{tilde Q} by (\ref{K-associate}-\ref{K2-associate}).
\arxiv{Indeed,
\begin{multline*}
    \kk e^{s x}\widetilde{\pp}_t f \stackrel{\eqref{PP}}{=} \kk\left(e^{sx}\left(\kk^{-1}e^{-t u^2}\kk\right) f\right)
   \stackrel{\eqref{K-associate}}{=} \kk \left(e^{s x} \kk^{-1}\left(e^{-t u^2}\kk f\right)\right)
     \\ \stackrel{\eqref{K2-associate}}{=}  \left(\kk e^{s x}\kk^{-1}\right)e^{-t u^2}\kk f \stackrel{\eqref{tilde Q}}{=}
     \widetilde{\qq}_s e^{-t u^2}\kk f.
\end{multline*}
}

Define $$F_{n+1}(x)=e^{(\uC-s_n) x}=e^{-s_n x}H_\tau(x),$$
and for $1\leq m\leq n$  let %
\begin{align}\label{Fm}
  F_m(x)& =e^{(s_m-s_{m-1}) x}H_{t_m}(x) \e \Big[e^{\sum\limits_{k=m+1}^n (s_{k}-s_{k-1})X_{t_k}}  e^{-s_n X_\tau}\Big\vert  X_{t_m}=x \Big]  \\
  &= e^{(s_m-s_{m-1}) x} \widetilde{\pp}_{t_{m+1}-t_m} [F_{m+1}].\nonumber
\end{align}
Since $F_{n+1}\in L_1(K_0)$,  by Lemma \ref{L0}  all $F_m\in L_1(K_0)$.

Next, we define
\begin{equation}
    \label{ind0} G_{n+1}(u):=h_{s_n}(u)=\kk [e^{(\uC-s_n) x}](u)=\kk [F_{n+1}](u),
\end{equation}
see \eqref{h_s}, and for
$1\leq m\leq n$  %
let
\begin{multline}\label{Gm}
G_m(u)= h_{s_{m-1}}(u) \e\left[  e^{-\sum_{k=m}^n (t_{k+1}-t_k)Z_{s_k}^2} \middle \vert Z_{s_{m-1}}=u\right] \\  = \widetilde{\qq}_{s_{m}-s_{m-1}}  \left[e^{-(t_{m+1}-t_{m})u^2}G_{m+1}(u)\right].
\end{multline}

Suppose that for some $1\leq m\leq n$ we have \begin{equation}
    \label{ind.} \kk F_{m+1}=G_{m+1}.
\end{equation}
(Note that \eqref{ind.} holds for $m=n$ by \eqref{ind0}.) Then
\begin{align*} %
\kk F_m & \stackrel{\eqref{Fm}}{=} \kk  e^{(s_m-s_{m-1}) x} \widetilde{\pp}_{t_{m+1}-t_m} [F_{m+1}]
\\&
\stackrel{\eqref{e-interwening}}{=} \widetilde\qq_{s_m-s_{m-1}}\, e^{-u^2(t_{m+1}-t_m)}\kk[F_{m+1} ]  \\ & \stackrel{\eqref{ind.}}{=}    \widetilde\qq_{s_m-s_{m-1}}\, e^{-u^2(t_{m+1}-t_m)} [G_{m+1}]
\stackrel{\eqref{Gm}}{=} G_m.
\end{align*}
This shows that $\kk F_1=G_1$.
 We now note that
     \begin{align*}
  \mathsf  F(\tilde x)&=e^{s_0\tilde x}   \e \Big[e^{\sum\limits_{k=1}^n (s_{k}-s_{k-1})X_{t_k}}  e^{-s_n X_\tau}\Big\vert X_0=\tilde x\Big]
   \\ &=e^{s_0\tilde x}   \pp_{t_0,t_1}  \left[ e^{(s_1-s_0)x }\e \Big[e^{\sum\limits_{k=2}^n (s_{k}-s_{k-1})X_{t_k}}  e^{-s_n X_\tau}\Big\vert  X_{t_{1}}=x \Big]\right](\tilde x)
   \\=&e^{s_0\tilde x} \frac{1}{H_0(\tilde x)}\widetilde \pp_{t_1}\left[H_{t_1}(x)  e^{(s_1-s_0)x}\e \Big[e^{\sum\limits_{k=2}^n (s_{k}-s_{k-1})X_{t_k}}  e^{-s_n X_\tau}\Big\vert  X_{t_{1}}=x \Big]\right](\tilde x)
\\
&=e^{s_0\tilde x} \frac{1}{H_0(\tilde x)}\widetilde \pp_{t_1}[F_1](\tilde x).
   \end{align*}
   So using   \eqref{PP} and (\ref{K-associate}-\ref{K2-associate}) again,
 we get
\begin{multline*}
  \label{F1F2}  e^{-s_0x}H_0(x)\mathsf F(x)=\kk^{-1} e^{-t_1u^2}\kk F_1(x)\stackrel{\eqref{ind.}}{=}\kk^{-1} [e^{-t_1u^2} G_1(u)](x) \\ =
 \kk^{-1} [e^{-t_1u^2} h_{s_0}(u) \mathsf G(u)](x).
\end{multline*}
Thus \eqref{Alexeys-F-G} holds.%
\end{proof}

Recall that  semigroup $\qq_{s,t}$  is well defined for all $-\vA<s<t<\uC$. The interval is non-empty when $\vA+\uC>0$, but it may contain negative numbers.

\begin{theorem}\label{T2}   Fix $\vA+\uC>0$.  Let $-\vA<s_0<s_1<\dots<s_{n}<\uC$   and $0=t_0<t_1<t_2<\dots<t_{n}<t_{n+1}=\tau$. Let $(X_t)_{t\in[0,\tau]}$ be the Markov process with finite dimensional distributions \eqref{joint-law0}
 and let  $(Z_s)_{s\in(-\vA, \uC)}$ be the Markov  %
 process
 described %
 at the end of Section   \ref{Sec:Z}.
 Then
\begin{equation}\label{For-proof}
\e\Big[ e^{\sum\limits_{k=0}^n - s_{k} (X_{t_{k+1}}-X_{t_k}) } \Big]
=\int_0^\infty\e\left[e^{-\sum\limits_{k=0}^{n} (t_{k+1}-t_k) Z_{s_k}^2}\middle\vert  Z_{s_0}=u\right] \widetilde \nu_{s_0}(\d u),
\end{equation}
where the $\sigma$-finite measure
$\widetilde\nu_{s_0}(\d u)$ %
is given by \eqref{Z-ini}.
\end{theorem}
\begin{proof}
We  apply %
\eqref{Alexeys-F-G} and then Parseval identity \eqref{Prsvl+} to %
the left hand side of \eqref{For-proof}. Recall notation \eqref{F-W} and \eqref{X-ini}.
 We have
\begin{multline*}
\e\Big[ e^{\sum\limits_{k=0}^n - s_{k} (X_{t_{k+1}}-X_{t_k}) } \Big]
  \stackrel{\eqref{F-W}}{=}\e \left[\mathsf F(X_0)\right]
  =\frac{1}{ C_{\vA,\uC}\topp\tau}\int_\r H_0(x)e^{\vA x} \mathsf F(x)\d x
  \\=\frac{1}{C_{\vA,\uC}\topp\tau}\int_\r e^{(\vA+s_0)x}  e^{-s_0x}H_0(x)\mathsf F(x)\d x
 \\ \stackrel{\eqref{Alexeys-F-G}}{=} \frac{1}{C_{\vA,\uC}\topp\tau}\int_\r e^{(\vA+s_0)x} \mathcal{K}^{-1}\left[ e^{-t_1u^2}h_{s_0}(u)\e\big[
    e^{-\sum_{k=1}^n (t_{k+1}-t_k)Z_{s_k}^2} \big \vert {Z_{s_0}=u}  \big]\right] \d x
 \\
  \stackrel{\eqref{Prsvl+}}{=}  \tfrac{1}{C_{\vA,\uC}\topp\tau}\int_0^\infty \mathcal{K}\left[e^{(\vA+s_0)x}\right] e^{-(t_1-t_0)u^2}h_{s_0}(u)\e\big[
    e^{-\sum_{k=1}^n (t_{k+1}-t_k)Z_{s_k}^2} \big \vert {Z_{s_0}=u}  \big] \mu(\d u)
   \\    \stackrel{\eqref{gsu}}{=}   \frac{1}{C_{\vA,\uC}\topp\tau}\int_0^\infty   h_{s_0}(u)g_{s_0}(u)  \e\left[
    e^{-\sum_{k=0}^n (t_{k+1}-t_k)Z_{s_k}^2} \middle \vert {Z_{s_0}=u}  \right] \mu(\d u)\nonumber
    \\  \stackrel{\eqref{Z-ini}}{=}  \int_0^\infty    \e\left[
    e^{-\sum_{k=0}^n (t_{k+1}-t_k)Z_{s_k}^2} \middle \vert {Z_{s_0}=u}  \right] \widetilde \nu_{s_0}(\d u).
\end{multline*}
\end{proof}

 We remark that by  Lemma \ref{Rem:Yizao}   we can replace the right hand side of \eqref{For-proof} by
 $$
 \int_0^\infty\e\left[e^{-\sum\limits_{k=0}^{n} (t_{k+1}-t_k) Z_{s_k}^2}\middle\vert  Z_{s_*}=u\right] \widetilde \nu_{s_*}(\d u)
 $$
for an arbitrary $s_*\in(-\vA,s_0)$.

\section{Relation to  Hartman-Watson density }\label{Sect: Rel-HW}
The (unnormalized) Hartman-Watson density function $t\mapsto \theta(r,t)$ is defined as
\begin{equation*}\label{formula_theta}
\theta(r,t)=\frac{re^{\pi^2/(2t)}}{\sqrt{2\pi^3 t}} \int_0^{\infty} e^{-y^2/(2t)-r \cosh(y)}
\sinh(y) \sin(\pi y/t) \d y, \;\;\; t>0, \; r>0.
\end{equation*}
In some papers this function is denoted by $\theta_r(t)$, but here we will follow the notation from \cite{MatsumotoYor2005I,MatsumotoYor2005II}.
An alternative integral representation was obtained by Yakubovich in \cite[formula (3.2),  slightly corrected]{Yakubovich2013}
\begin{equation}\label{Yak-theta}
\theta(r,t)=\frac{1}{\pi} \int_0^{\infty} e^{- t u^2/2} K_{\i u}(r)  \frac{\d u}{\vert \Gamma(\i u)\vert ^2}.
\end{equation}

The scaled version of this density function was introduced by Hartman and Watson in \cite{hartman1974normal} and the function
$\theta(r,t)$ was studied extensively by Marc Yor and many other researchers (see \cite{MatsumotoYor2005I,MatsumotoYor2005II} and the references therein). The interest in the function $\theta(r,t)$ can be explained by it close relation to the exponential functionals of Brownian motion and to pricing of Asian options in the Black--Scholes model. Let $B=\{B_t\}_{t\ge 0}$ be a one-dimensional Brownian motion starting from zero and define $B^{(\mu)}=\{B^{(\mu)}_t=B_t+\mu t\}_{t\ge 0}$ to be the Brownian motion with drift $\mu \in \r$. Define
$$
A^{(\mu)}_t=\int_0^t e^{2 B_s^{(\mu)}} \d s, \; \; \; t \ge 0.
$$
Then for $t>0$, $y>0$ and $x\in \r$ it holds that
\begin{equation*}
\p\left(A^{(\mu)}_t \in \d y, B^{(\mu)}_t \in \d x\right) = e^{\mu x - \mu^2 t/2}
\exp\big( -(1+e^{2x})/(2y)\big) \theta(e^x/y,t) \frac{\d y \d x}{y},
\end{equation*}
see \cite[Theorem 4.1]{MatsumotoYor2005I}.
The function $\theta(r,t)$ has the following Laplace transforms (see Proposition 4.2 and Theorem A.1 in \cite{MatsumotoYor2005I})
\begin{align}
&\int_0^{\infty} e^{-x r} \theta(r,t) \frac{\d r}{r}=\frac{1}{\sqrt{2\pi t}} \label{MY2}
\exp\Big(-{\textrm{Argcosh}}(x)^2/(2t) \Big), \;\;\; x\ge 1, \; t>0,\\
&\int_0^{\infty} e^{-\la^2 t/2}  \theta(r,t) \d t=I_{\la}(r), \;\;\; \la>0, \; r>0,  \label{MY1}
\end{align}
where $I_{\la}$ is the modified Bessel function of the first kind.
The Yakubovich heat kernel $p_t(x,y)$ can be expressed in terms of $\theta(r,t)$ as follows
\begin{equation}\label{Yak-Hart-Wat}
p_t(x,y)= \int_{\r}\, \exp\left(-\tfrac12(e^{r+x-y}+e^{r+y-x}+e^{x+y-r})\right)\theta(e^r,2t)\,\d r.
\end{equation} %
The above identity is a simple re-write of
\cite[formula (50)]{SousaYakubovich2018}, and can also be deduced  from a more general formula in \cite[Remark 4.1]{MatsumotoYor2005I} which relates the kernel to the so-called Liouville potential. It makes positivity of $p_t(x,y)$ obvious.

Yakubovich's formula  \eqref{Yak-Hart-Wat} gives   an unexpected   closed-form formula for the Laplace transform %

$$L_\la(\vA,\uC)= \int_0^\infty e^{-\la^2 \tau} C_{\vA,\uC}\topp \tau \d \tau$$
of
the normalizing constant \eqref{C-formula}.
\begin{theorem}\label{Prop-LapC}
  If $0<\vA+\uC<2$ %
  then for $\la>\max\{-\uC,-\vA\}$ we have %

 \begin{equation*}
  \label{LapTC}
L_\la(\vA,\uC)=   \frac{2^{\vA+\uC}\pi }{8\sin (\pi\frac{\vA+\uC}{2})} \frac{\Gamma(\frac{\vA+\la}{2}, \frac{\uC+\la}{2})}{\Gamma(\frac{\la+2-\vA}{2},\frac{\la+2-\uC}{2})}.
\end{equation*}
\end{theorem}

  \begin{proof}
Putting  \eqref{Yak-Hart-Wat} into  \eqref{C-formula}, and noting that the integrand is positive, we write
    \begin{equation*}
     C_{\vA,\uC}\topp \tau =  \int_\r\, {\mathcal I}(r) \, \theta(e^r,2\tau)
\,\d r,
  \end{equation*}
  where  %
  \begin{align*}\label{Ir-def}
{\mathcal I}(r)&:=\int_{\r} \int_{\r} e^{\vA x+\uC y} \exp\left(-\tfrac12(e^{r+x-y}+e^{r+y-x}+e^{x+y-r})\right)   \d x \d y.
 \end{align*}
 We will need the following two identities
\cite[3.471.9 and 6.596.3]{gradshteyn2007table} or \cite[6.3(17) and  6.8(32)]{erdelyi1954fg} :
\begin{equation}\label{id1}
\int_0^{\infty} x^{\nu-1} e^{-\beta/x-\gamma x} \d x=2 (\beta/\gamma)^{\nu/2} K_{\nu}(2\sqrt{\beta \gamma}), \;\;\; \re(\beta)>0, \; \re(\gamma)>0,
\end{equation}
and
\begin{equation}\label{id2}
\int_{0}^{\infty} K_{\nu}(\sqrt{x^2+z^2}) \frac{x^{2\mu+1}}{(x^2+z^2)^{\nu/2}} \d x
=\frac{2^{\mu} \Gamma(\mu+1)}{z^{\nu-\mu-1}} K_{\nu-\mu-1}(z), \;\;\; \re(\mu)>-1.
\end{equation}
We do a change of variables $e^x=w$, $e^y=z$, $e^r=R$ and obtain
\begin{align*}
{\mathcal I}(r)&=\int_0^{\infty} \int_0^{\infty} w^{\vA-1} z^{\uC-1}
 \exp\left(-\tfrac12(Rw/z+Rz/w+zw/R)\right) \d z \d w.
 \end{align*}
 The inner integral is computed via \eqref{id1} as follows
 \begin{align*}
  \int_0^{\infty} z^{\uC-1}
 \exp\left(-\tfrac12(w/R+R/w)z-\tfrac12 Rw/z\right) \d z=
 2 R^\uC w^\uC \frac{K_{\uC}(\sqrt{R^2+w^2})}{(R^2+w^2)^{\uC/2}}.
 \end{align*}
 The above formula holds for all real $\uC$.
 Next we compute the outer integral using \eqref{id2}
 \begin{align}\label{id3}
 {\mathcal I}(r)&=2 R^\uC \int_0^{\infty} w^{\vA+\uC-1}
   \frac{K_{\uC}(\sqrt{R^2+w^2})}{(R^2+w^2)^{\uC/2}} \d w
  =2 R^\uC \frac{2^{\frac{\vA+\uC}{2}-1} \Gamma\Big(\frac{\vA+\uC}{2}\Big)}{R^{(\uC-\vA)/2}} K_{(\uC-\vA)/2}(R)
  \\&=2^{\frac{\vA+\uC}{2}} \Gamma\Big(\frac{\vA+\uC}{2}\Big) e^{r(\vA+\uC)/2} K_{(\uC-\vA)/2}(e^r).\nonumber
\end{align}

 From \eqref{id3} and \eqref{MY1} we get
  \begin{multline*}
     \int_0^\infty e^{-\la^2 \tau} C_{\vA,\uC}\topp\tau \d \tau
      \\ =2^{\frac{\vA+\uC}{2}} \Gamma\Big(\frac{\vA+\uC}{2}\Big)\,\int_\r\,e^{r(\vA+\uC)/2}\,\left(\int_0^\infty e^{-\la^2 \tau }  \theta(e^r,2\tau) \d \tau\right) K_{(\uC-\vA)/2}(e^r) \d r
     \\ =2^{\frac{\vA+\uC}{2}-1} \Gamma\Big(\frac{\vA+\uC}{2}\Big)  \int_\r  e^{r(\vA+\uC)/2} I_\la(e^r) K_{(\uC-\vA)/2}(e^r) \d r \\
      =2^{\frac{\vA+\uC}{2}-1} \Gamma\Big(\frac{\vA+\uC}{2}\Big)  \int_0^\infty  x^{\frac{\vA+\uC}{2}-1} I_\la(x)  K_{(\uC-\vA)/2}(x) \d x,
  \end{multline*}
  with substitution $x=e^r$.
  Formula \cite[6.8(43)]{erdelyi1953higher}
  $$
  \int_0^\infty x^{s-1} I_\nu(x)K_\mu(x)\d x= \frac{\Gamma(\frac{s+\mu+\nu}{2})B(1-s,\frac{s+\nu-\mu}{2})}{2^{2-s} \Gamma(\frac{\mu+\nu-s}{2}+1)},
  $$ which holds if $\re(-\nu\pm\mu)<\re (s) <1$, used with $s=(\vA+\uC)/2$, $\mu=(\uC-\vA)/2$, $\nu=\la$
  gives
  $$
     \int_0^\infty e^{-\la^2 \tau} C_{\vA,\uC}\topp\tau \d \tau = \frac{2^{\vA+\uC}}{8} \frac{\Gamma(\frac{\vA+\la}{2},\frac{2-\vA-\uC}{2},\frac{\uC+\la}{2})}{\Gamma(\frac{\la+2-\vA}{2},\frac{\la+2-\uC}{2})}\Gamma(\frac{\vA+\uC}{2}).$$
   This simplifies by Euler's reflection formula,
   $$
   \int_0^\infty e^{-\la^2 \tau} C_{\vA,\uC}\topp\tau \d \tau =  \frac{\pi 2^{\vA+\uC}}{8\sin (\pi\frac{\vA+\uC}{2})} \frac{\Gamma(\frac{\vA+\la}{2}, \frac{\uC+\la}{2})}{\Gamma(\frac{\la+2-\vA}{2},\frac{\la+2-\uC}{2})}.$$
  \end{proof}

 \begin{remark}
 \cite[Theorem 3.1]{craddock2014integral} gives a formula for Yakubovich heat kernel that in our notation simplifies to:
$$
p_t(x,y)=\frac{1}{4\sqrt{\pi}t^{3/2}} \int_{\vert x-y\vert }^\infty u e^{-\frac{u^2}{4t}} J_0(e^{x+y+u}+e^{x+y-u}- e^{2x}-e^{2y})\d u,
$$
where $J_0$ is the Bessel function of the first kind.
For another formula of similar type, see \cite[Proposition 5.4]{MatsumotoYor2005I}.
 \end{remark}

\arxiv{
\subsection{Using GIG and gamma laws} Some integral  identities can be replaced by more probabilistic arguments that rely on special properties of Generalized Inverse Distributions (GIG).
Here we show how to re-derive \eqref{id3} by this technique.
\newcommand{\R}{\r}
Let $\mathrm{GIG}(p,\alpha,\beta)$, $p\in\R$, $\alpha,\beta>0$, be the generalized inverse Gaussian distribution %
defined by the density %
\begin{equation*}\label{GIG}
f(y)=\left(\tfrac{\alpha}{\beta}\right)^p\tfrac{y^{p-1}}{2K_p(\sqrt{\alpha\beta})}\,e^{-\tfrac{\alpha y}{2}-\tfrac{\beta }{2y}}\,I_{(0,\infty)}(y).
\end{equation*}
Let $\mathrm G(q,\gamma)$, $q,\gamma>0$, be the gamma distribution defined by the density
\begin{equation*}\label{gamm}
g(x)=\tfrac{\gamma^q}{\Gamma(p)}\,x^{q-1}\,e^{-\gamma x}\,I_{(0,\infty)}(x).
\end{equation*}

Let $X\sim \mathrm G(q,\tfrac{1}{2c})$ and $Y\sim \mathrm{GIG}(p,c,c)$ be independent.

Consider a function $\psi$ defined by $$\psi(x,y)=\left(\sqrt{xy},\,\sqrt{\tfrac{y}{y}}\right),\quad x,y>0.$$ Note that this is a diffeomorphism from $(0,\infty)^2$ onto itself with the Jacobian of $\psi^{-1}(z,w)=(zw,\tfrac{z}{w})$ of the form
$$
J_{\psi^{-1}}(z,w)=\tfrac{2z}{w}.
$$
Therefore the density of the random vector $$(Z,W)=\left(\sqrt{XY},\,\sqrt{\tfrac{X}{Y}}\right)$$ is
\begin{align*}
f_{Z,W}(z,w)&=J_{\psi^{-1}}(z,w)f_X(zw)f_Y(z/w)\\
&=
\tfrac{2z}{w}\,\tfrac{1}{(2c)^q\Gamma(q)}\,(zw)^{q-1}e^{-\tfrac{zw}{2c}}\,\tfrac{1}{2K_p(c)}\,(z/w)^{p-1}\,e^{-\tfrac{c(z/w)}{2}-\tfrac{c}{2(z/w)}}\\
&=\tfrac{z^{q+p-1}w^{q-p-1}\,\exp\left(-\tfrac{zw}{2c}-\tfrac{cz}{2w}-\tfrac{cw}{2z}\right)}{C(p,q,c)},
\end{align*}
where $C(p,q,c)=(2c)^q\Gamma(q)K_{p}(c)$ is the normalizing constant.

Thus the density of  $(U,V)=(\log\,Z,\,\log W)$ is
$$
f_{U,V}(u,v)=\tfrac{e^{(p+q)v+(q-p)u}\,\exp\left(-\tfrac{1}{2}\left(c^{-1}e^{u+v}+c e^{v-u}+ce^{u-v}\right)\right)}{C(p,q,c)},\quad u,v\in\R.
$$

Summing up, $\mathcal I(r)$, 
is the normalizing constant of the joint density of $$(U,V)=\tfrac{1}{2}(\log X +\log Y,\,\log X-\log Y)$$
where $X\sim \mathrm G(\tfrac{\vA+\uC}{2},\,\tfrac{1}{2e^r})$ and $Y\sim\mathrm{GIG}(\tfrac{\vA-\uC}{2},e^r,e^r)$ are independent random variables.
Thus
$$
\mathcal I(r)=C\left(\tfrac{\vA-\uC}{2},\tfrac{\vA+\uC}{2},e^r\right)=(2e^r)^{\tfrac{\vA+\uC}{2}}\,\Gamma(\tfrac{\vA+\uC}{2})\,K_{\tfrac{\vA-\uC}{2}}(e^r)
$$
as given in \eqref{id3}.

Properties of GIG can also replace  Macdonald's formula in the proof of \eqref{Yak-Hart-Wat}.

 }

\appendix

\section{Additional properties of Kontorovich--Lebedev transform}\label{Sec: Additional Props}
We  collect here facts about   $\kk$ and $\kk^{-1}$ that we need in the paper.
\begin{lemma}\label{L0}
\begin{enumerate}[(i)]
\item  \label{L0i}
If $F$ is in $L_1(K_0)$   then $\kk F$ is bounded.
\item \label{L0-ii} If $G$ is in $L_1(\d \mu)$  and $\eps>0$, then  function
$
e^{\eps x} \kk^{-1} G
$ is in $L_1(K_0)$.
\item \label{L0iii}
If $F$ is in $L_1(K_0)$  and $G$ is in $L_1(\d \mu)$ then we have Parseval's identity
  \begin{equation}\label{Prsvl+}
    \int_\r F(x)\mathcal{K}^{-1}[G](x) \d x =  \int_0^\infty G(u) \mathcal{K}[F](u) \mu(\d u).
  \end{equation}
\item \label{L0-iv}
If $F$ is in $L_1(K_0)$ and $\delta >0$, then we have associativity
\begin{equation}
  \label{K-associate}
  \left(\kk^{-1}e^{-\delta  u^2} \kk\right)F=\kk^{-1}\left(e^{-\delta  u^2}\kk F\right).
\end{equation}
\item \label{L0-v}
If $G$ is in $L_1(\d \mu)$ and $\delta >0$, then we have associativity
\begin{equation}
  \label{K2-associate}
  \left(\kk e^{\delta x} \kk^{-1}\right)G=\kk \left(e^{\delta x}\kk^{-1} G\right).
\end{equation}
\end{enumerate}

\end{lemma}

\begin{proof}[Proof of \eqref{L0i}] This follows from $\vert K_{\i u}(e^x)\vert \leq K_0(e^x)$.
\end{proof}
\begin{proof}[Proof of \eqref{L0-ii}]

This follows from $
\vert \kk^{-1}(G)\vert =\left\vert \int_0^\infty K_{\i u}(e^x) G(u)\mu(\d u)\right\vert \leq C K_0(e^x)$ and from the well known bounds
\begin{equation}
  \label{K0-bd1}K_0(x)\leq K_{1/2}(x)=  e^{-x}\frac{\sqrt{\pi}}{\sqrt{2x}}\leq C e^{-x}
\end{equation} for $x\geq 1$,
and
\begin{equation}
  \label{K0-bd2} K_0(e^x)\leq C_\eps e^{-\eps x/4}
\end{equation}
 for $x<0$; the latter is a consequence of   the trivial bound  $\cosh t\geq t^{2k}/(2k)!$ applied to \eqref{K-def} with $u=0$ and large enough $k$.

 Thus $e^{\eps x} K_0^2(x)\leq C_\eps^2\exp (\eps x/2)1_{x<0}+ K_0(0)+C^2\exp(-2\exp(x))1_{x>0} $  is integrable, i.e. $e^{\eps x} K_0(x)\in L_1(K_0)$.
\end{proof}
\begin{proof}[Proof of \eqref{L0iii}] (See \cite[Lemma 2.4]{yakubovich1996index} and the discussion therein.) Since
 $$\int_\r \int_0^\infty \vert F(x)K_{\i u}(e^x) G(u)\vert  \d x \mu(\d u)\leq \int_\r  \vert F(x)\vert K_{0}(e^x) \d x \int_0^\infty\vert G(u)\vert \mu(\d u)<\infty,$$
we  can apply Fubini's theorem to interchange the order of integrals:
 $$\int_\r F(x) \left(\int_0^\infty \,K_{\i u}(e^x) G(u)  \mu(\d u)\right) \d x = \int_0^\infty  \left(\int_\r F(x) K_{\i u}(e^x) \d x\right) G(u)  \mu(\d u),$$
 which is \eqref{Prsvl+}.
\end{proof}
\begin{proof}[Proof of \eqref{L0-iv}]
Identity \eqref{K-associate} in the integral form consists of change of order in the integrals:
\begin{multline*}
\int_\r \left(\int_0^\infty e^{-\delta u^2} K_{\i u}(e^x)K_{\i u}(e^y)\mu(\d u)\right) F(y)\d y
\\=
\int_0^\infty K_{\i u}(e^x) e^{-\delta u^2}\left( \int_\r  K_{\i u}(e^y)F(y)\d y \right)\mu(\d u).
\end{multline*}
The integrand is dominated by $K_0(e^x)K_0(e^y)\vert F(y)\vert e^{-\delta u^2}$. From   \eqref{mu} we see that $e^{-\delta u^2}\in L_1(\d \mu)$,  so with  $\vert F\vert \in L_1(K_0)$, this justifies the use of Fubini's theorem.
\end{proof}
\begin{proof}[Proof of \eqref{L0-v}]
Identity \eqref{K2-associate} in the integral form is
\begin{multline*}
   \int_0^\infty   G(v) \left(\int_\r e^{sx}K_{\i u}(e^x)K_{\i v} (e^x) \d x\right)\mu(\d v) \\=\int_\r  K_{\i u}(e^x)e^{\delta x} \left(\int_0^\infty K_{\i v}(e^x)G(v)\mu(\d v)\right) \d x.
\end{multline*}
The use of Fubini's theorem is again justified by the fact that
$e^{\delta x} K_0(e^x)$ is in $L_1(K_0)$; the latter follows from the  bounds  stated in the proof of \eqref{L0-ii}.
\end{proof}

\section{Semigroup property for $\qq_{s,t}$}\label{Sect:SemiQ}
In this section we prove Proposition \ref{Prop-Semi}.
The fact that
$$\int_0^\infty q_{s,t}(u,v)\d v=1$$
is a consequence of Beta integral
\cite[(7.i)]{Askey:1989} or \cite[Section 1.3]{koekoek1998askey}
   for the weight function of the continuous dual Hahn polynomials
\begin{equation}\label{WI0}
\int_0^\infty \frac{\prod_{j=1}^3 \left(\Gamma(\uC_j+\i{x})\Gamma(\uC_j-\i{x})\right)
}{\left\vert \Gamma(2\i{x})\right\vert ^2} dx
=2 \pi  \prod_{1\leq k<j\leq 3}\Gamma (\uC_k+\uC_j).
\end{equation}
We use \eqref{WI0} with $$x=v/2, \uC_1=(\uC-t)/2,\uC_2=(t-s+\i u)/2, \uC_3=(t-s-\i u)/2.$$

The  Chapman-Kolmogorov equations
\begin{equation*}\label{CK-eq} \int_0^\infty q_{r,s}(u,v) q_{s,t}(v,w)\d v=  q_{r,t}(u,w), \quad  r<s<t,
\end{equation*}
are a consequence of de Branges-Wilson integral
   \cite{deBranges:1972,Wilson:1980}
\begin{equation*}\label{WI}
\int_0^\infty \frac{\prod_{j=1}^4 \left(\Gamma(\uC_j+\i{x})\Gamma(\uC_j-\i{x})\right)
}{\left\vert \Gamma(2\i{x})\right\vert ^2} dx
=\frac{2 \pi  \prod_{1\leq k<j\leq 4}\Gamma (\uC_k+\uC_j)}{\Gamma (\uC_1+\uC_2+\uC_3+\uC_4)},
\end{equation*}
after we note that
   $$\frac{q_{r,s}(u,x) q_{s,t}(x,w)}{q_{r,t}(u,w)}=  \frac{\Gamma (\uC_1+\uC_2+\uC_3+\uC_4)}{4 \pi  \prod_{1\leq k<j\leq 4}\Gamma (\uC_k+\uC_j)}\frac{\prod_{j=1}^4 \left(\Gamma(\uC_j+\i{x}/2)\Gamma(\uC_j-\i{x}/2)\right)
}{\left\vert \Gamma(\i{x})\right\vert ^2}. $$
with $$\uC_1=(s-r+\i u)/2,\; \uC_2=(s-r-\i u)/2,\; \uC_3=(t-s+\i w)/2,\; \uC_4=(t-s-\i w)/2.$$
Related algebraic identity is used for similar purposes
in \cite[Section 3]{Bryc:2009} and in \cite[Section 7.3]{CorwinKnizel2021}.
\section{Continuous dual Hahn process}\label{Sec:ExtendedCDH}
To define the continuous dual Hahn process,  Corwin and Knizel \cite{CorwinKnizel2021}   specify its marginal distributions as a $\sigma$-finite measure on $\r$ and transition probability
distributions, and then show that they are consistent.  For technical reasons, they define the process only on the interval $[0,S)$, where $S=2+(\uC-2)1_{(0,2)}(\uC)$. For %
Theorem \ref{T1.3},
we need  this process to be extended to $s\in [0,\uC)$ with $\uC>0$. For %
Theorem \ref{T1.2},
we need  this process for $s\in (-\vA,\uC)$. %

The following summarizes Ref. \cite{Bryc-2021}, restricted to   the  time domain, $(-\infty,\uC]$, where $\uC\in\r$.
To define the family of transition probabilities  we use the orthogonality probability measure
for the continuous dual Hahn orthogonal polynomials. These are monic polynomials which depend
on three parameters
traditionally denoted by $a,b,c$, but  to avoid confusion with the parameters $\vA, \uC$ above, we will denote them by $\alpha,\beta,\gamma$.

We take $\alpha\in\r$ and $\beta,\gamma$ are either real or complex conjugates.
Let
\begin{equation*}
  \label{AnCn}
  A_n=(n+\alpha+\beta)(n+\alpha+\gamma),\quad C_n=n(n-1+\beta+\gamma).
\end{equation*}
The continuous dual Hahn orthogonal polynomials in real variable $x$ are
defined by the
three step recurrence relation
\begin{equation}\label{Recursion}
x p_n(x)= p_{n+1}(x)+(A_n+C_n-\alpha^2)p_n(x)+A_{n-1}C_n p_{n-1}(x),
\end{equation}
with the usual initialization $p_{-1}(x)=0$, $p_0(x)=1$.
Compare \cite[(1.3.5)]{koekoek1998askey}.

 Recursion \eqref{Recursion} defines a family of monic polynomials. %
Favard's theorem, as stated in  \cite[Theorem A.1]{Bryc-Wesolowski-08},    allows us to recognize orthogonality from the condition  that
\begin{equation}
  \label{Favard}
  \prod_{k=0}^nA_kC_{k+1}\geq 0 \mbox{ for all $n$}.
\end{equation}
If parameters  $\alpha,\beta,\gamma$ are such that \eqref{Favard} holds,
then
there exists a unique (see \cite[Lemma 2.1]{Bryc-2021}) probability measure
$\rho(\d x\vert \alpha,\beta,\gamma)$ such that for $m\ne n$ we have
 $$\int_\r p_n(\tfrac{x}{4}\vert \alpha,\beta,\gamma)  p_m(\tfrac{x}{4}\vert \alpha,\beta,\gamma)\rho(\d x\vert \alpha,\beta,\gamma)=0.$$
 (Notice the factor $x/4$ to match the scaling in \cite{CorwinKnizel2021}. This factor is not present in \cite{Bryc-2021}, so we need to recalculate some of the results that we quote from there.)

When defined, probability measures $\rho(\d x\vert \alpha,\beta,\gamma)$  may be discrete, continuous, or of mixed type, depending on the parameters.
We use them to define transition probabilities for $-\infty<s<t\leq \uC$ as follows
\begin{equation}\label{T-trans}
\mathfrak p_{s,t}(x,\d y)=\begin{cases}
\rho(\d y\vert \tfrac{\uC-t}{2},\tfrac{t-s - \i \sqrt{x}}{2},\tfrac{t-s + \i \sqrt{x}}{2}) & x\geq 0, \\
\rho(\d y\vert \tfrac{\uC-t}{2},\tfrac{t-s -  \sqrt{-x}}{2},\tfrac{t-s +  \sqrt{-x}}{2})  & -(\uC-s)^2\leq x<0, \\
\delta_{-(\uC-t)^2} & x<-(\uC-s)^2.
\end{cases}
\end{equation}
By \cite[Proposition 3.1]{Bryc-2021}, these measures are well defined. A version of
\cite[Theorem  3.5(i)]{Bryc-2021} says that Chapman-Kolmogorov equations hold:
for a Borel subset $U\subset \r$,  $ x\in\r$, and $-\infty<s<t<u\leq \uC$ we have
\begin{equation}\label{C-K}
  \mathfrak p_{s,u}(x,U)=\int_\r \mathfrak p_{t,u}(y,U)  \mathfrak p_{s,t}(x,\d y).
\end{equation}
\begin{definition}\label{Def.C1}
We denote by  $(\TT_t)_{-\infty<t\leq \uC}$ the Markov
 process
on state space $\r$, with transition probabilities \eqref{T-trans}.
\end{definition}

Next we follow  \cite[Definition 7.8]{CorwinKnizel2021} and introduce a family of $\sigma$-finite positive measures that are the entrance law \cite[pg. 5]{Sharpe-1988} for the Markov
process
$(\TT_t)$. For $\vA>-\uC$ and $s\leq \uC$, let
  \begin{multline}
    \label{their-nu}
   \mathfrak p_{s}(\d x)=
     \frac{(\vA+\uC)(\vA+\uC+2)}{16\pi} \frac{\vert \Gamma(\frac{s+\vA+\i \sqrt{x}}{2},\frac{\uC-s+\i \sqrt{x}}{2} )\vert ^2}{2\sqrt{x}\vert \Gamma(\i \sqrt{x})\vert ^2}1_{x>0}\;\d x
     \\ +
    \sum_{\{j:\; 2j+\vA+s<0\}} p_j(s) \delta_{-(2j+\vA+s)^2}(\d x),
  \end{multline}
 where for $j\in\mathbb{Z}\cap[0, -(\vA+s)/2) $  the discrete masses  are given by
  \begin{align}
    \label{their-atoms}
    p_j(s)
&    = \frac{(\vA+\uC)(\vA+\uC+2)}{4}\frac{\Gamma(\frac{\uC-\vA-2s}{2},\frac{\vA+\uC}{2})}{\Gamma(-\vA-s)}\cdot \frac{(\vA+2j+s)(\vA+s,\frac{\vA+\uC}{2})_j}{(\vA+s)j!(1+\frac{2s+\vA-\uC}{2})_j}.
  \end{align}
This is recalculated from  \cite[Definition 7.8]{CorwinKnizel2021}; %
the parameters in
\cite[(1.12)]{CorwinKnizel2021} are $u=\uC/2$, $v=\vA/2$.

\begin{remark}\label{R:no-atoms}  We note that there are no atoms when $-\vA<s<\uC$, so in this case $\mathfrak p_s(\d x)$ is absolutely continuous with the density supported on $(0,\infty)$. %
\end{remark}

The following summarizes properties of  measures $\mathfrak p_s$ and transition probabilities
$\mathfrak p_{s,t}(x,\d y)$ that we will need.
\begin{theorem} \label{TH:CDH} Suppose that $\vA+\uC>0$, %
and $U\subset \r$ is a Borel set.
\begin{enumerate}[(i)]
    \item Measures $\{\mathfrak p_s: -\infty< s<\uC\}$ are the entrance law for the
     continuous dual Hahn process, i.e., they satisfy
    \begin{equation}
        \label{p-inv}
        \mathfrak p_t(U)=\int_\r \mathfrak p_{s,t}(x,U) \mathfrak p_s(\d x), -\infty< s<t\leq\uC.
    \end{equation}
    \item       For $-\infty<s<t< \uC$  and $x>0$, transition probabilities for the
     continuous dual Hahn process  are absolutely continuous with density \eqref{p2q0}.

\end{enumerate}

\end{theorem}
\begin{proof}

(i) is a version of
\cite[Theorem 4.1 ]{Bryc-2021}

(ii) This is the case of  measure $\rho(\d x\vert \alpha,\beta,\gamma)$ with parameters
$\alpha>0$, $\beta=\bar \gamma$,  $\re(\beta)>0$, so the density can be read out from
\cite[Section 1.3]{koekoek1998askey}, see also \cite[Section 3]{Bryc:2009} and \cite[Section 1.3]{CorwinKnizel2021}.
\end{proof}

\begin{remark}\label{R-ac} By Remark \ref{R:no-atoms} and Theorem \ref{TH:CDH}(ii), transition probabilities \eqref{T-trans} with the entrance law \eqref{their-nu}
define an  absolutely continuous Markov process  $(\TT_t)_{-\vA<t<\uC}$  on the state space $(0,\infty)$. This is the process that appears in Theorem \ref{T1.2}.
However, if the state space is enlarged to include points $x<0$, then the transition probabilities may have a discrete component; for example, at $x=-(\uC-s)^2$,
transition probability $\mathfrak p_{s,t}(x,\d y)$ is a degenerate measure concentrated at $-(\uC-t)^2$, see \cite{Bryc-2021}. In particular,
if $s<-\vA$ then the entrance law \eqref{their-nu} %
 may    put some mass at points in $(-\infty,0)$.
\end{remark}

\arxiv{
\section{Proofs of Yakubovich identities}\label{Sect-Yakub-ids}
For completeness, we sketch the proofs of Yakubovich identities \eqref{Yak-theta} and \eqref{Yak-Hart-Wat}, omitting issues of convergence.
\begin{proof}[Sketch of proof of \eqref{Yak-theta}]
For real $u$ and $\la>0$, we have
$$\int_{-\infty}^\infty K_{iu}(e^x)I_\la(e^x) \d x=\frac{1}{\la^2+u^2},$$
where $I_\la(x)$ denotes the modified Bessel $I$ function.  Recalling Kontorovich-Lebedev transforms  $\kk$ and $\kk^{-1}$ \eqref{K-L+} and \eqref{K-inv}, this formula says
 $\kk [I_\la(e^x)]=\frac{1}{\la^2+u^2}$. Hence
 \begin{equation*}
  I_\la(e^x)=\kk^{-1}\frac{1}{\la^2+u^2}.
 \end{equation*}
  Writing this explicitly, %
  we have
\begin{align*}
  I_\la(e^r)&=\frac{2}{\pi} \int_0^\infty \frac{1}{\la^2+u^2} \frac{K_{\i u}(e^r)}{\vert \Gamma(\i u)\vert ^2}\d u
 =  \frac{2}{\pi} \int_0^\infty \int_0^\infty e^{-(\la^2+u^2) t} \d t \frac{K_{\i u}(e^r)}{\vert \Gamma(\i u)\vert ^2}\d u
 \\
& =\int_0^\infty e^{-\la^2 t}\left[\frac{2}{\pi} \int_0^\infty e^{-t u^2} \frac{K_{\i u}(e^r)}{\vert \Gamma(\i u)\vert ^2}\d u\right].\nonumber
\end{align*}
Comparing this with    \eqref{MY1}, we see that
\begin{equation*}\label{Last-straw}
  \frac{2}{\pi} \int_0^\infty e^{-t u^2} \frac{K_{\i u}(e^r)}{\vert \Gamma(\i u)\vert ^2}\d u =2 \theta(e^r,2t),
\end{equation*}
i.e., we have \eqref{Yak-theta}.
\end{proof}
}

\arxiv{
\begin{proof}[Sketch of proof of \eqref{Yak-Hart-Wat}]
 We use Macdonald's formula \cite[Vol II, 7.7.6]{erdelyi1954fg} for the product of the Bessel K functions with the same index:
\begin{equation}
  \label{Macdonalds}
    K_{\i u} (e^x)   K_{\i u} (e^y)=\frac12 \int_{-\infty}^\infty \exp[-\frac12(e^{r+x-y}+e^{r+y-x}+e^{x+y-r}) ]K_{\i u}(e^r) \d r .
\end{equation}

  Recall that $\pp _t=\kk^{-1} e^{-tu^2}\kk$, i.e.,
  \begin{align*}
  p_t(x,y)&=\frac{2}{\pi} \int_0^\infty e^{-tu^2} \frac{K_{\i u} (e^x)   K_{\i u} (e^y)}{\vert \Gamma(\i u)\vert ^2}\d u
  \\
  &\stackrel{\eqref{Macdonalds}}{=} \frac{2}{\pi} \int_0^\infty e^{-tu^2} \left[\frac12 \int_{-\infty}^\infty \exp[-\frac12(e^{r+x-y}+e^{r+y-x}+e^{x+y-r} )]K_{\i u}(e^r) \d r \right]\frac{1}{\vert \Gamma(\i u)\vert ^2}\d u\nonumber
  \\&
  \stackrel{\rm Fubini}{=}\int_{-\infty}^\infty \exp[-\frac12(e^{r+x-y}+e^{r+y-x}+e^{x+y-r}) ] \left[ \frac{1}{\pi} \int_0^\infty  e^{-tu^2} \frac{K_{\i u}(e^r)}{\vert \Gamma(\i u)\vert ^2} \d u\right] \d r\nonumber
  \\&\stackrel{\eqref{Yak-theta}}{=} \int_{-\infty}^\infty \exp[-\frac12(e^{r+x-y}+e^{r+y-x}+e^{x+y-r}) ] \theta(e^r,2t) \d r.\nonumber
  \end{align*}
\end{proof}
}


\subsection*{Acknowledgement}
This research project was initiated following WB's visit in 2018 to Columbia University (whose hospitality is greatly appreciated) and it would not have been possible without Ivan Corwin and Alisa Knizel generously sharing their preliminary results that later appeared in \cite{CorwinKnizel2021}.

We extend our thanks to Semyon Yakubovich for helpful comments on the early draft of this paper. We appreciate
information from Ofer Zeitouni  about the  reciprocal processes.
 We thank Guillaume Barraquand for Ref. \cite{barraquand2021steady} and the discussion of its contents.

WB's research was partially supported by Simons Foundation/SFARI Award Number: 703475. AK's research was partially supported by
The Natural Sciences and Engineering Research Council of Canada.
YW's research was partially supported by Army Research Office, US (W911NF-20-1-0139).
JW's research was partially supported by grant
2016/21/B/ST1/00005 of National Science Centre, Poland.


\end{document}